\documentclass[11pt,reqno]{amsart}

\usepackage[a4paper,margin=1.08in]{geometry}
\usepackage{amsmath,amssymb,amsthm,mathtools,mathrsfs}
\usepackage{enumitem}
\usepackage{booktabs,longtable,array}
\usepackage{microtype}
\usepackage{xcolor}
\usepackage[colorlinks=true,linkcolor=blue,citecolor=blue,urlcolor=blue]{hyperref}

\numberwithin{equation}{section}

\newtheorem{theorem}{Theorem}[section]
\newtheorem{proposition}[theorem]{Proposition}
\newtheorem{lemma}[theorem]{Lemma}
\newtheorem{corollary}[theorem]{Corollary}
\theoremstyle{remark}
\newtheorem{remark}[theorem]{Remark}

\newcommand{\R}{\mathbb R}
\newcommand{\N}{\mathbb N}
\newcommand{\cS}{\mathcal S}
\newcommand{\cP}{\mathcal P}
\newcommand{\cZ}{\mathcal Z}
\newcommand{\cE}{\mathcal E}
\newcommand{\supp}{\operatorname{supp}}
\newcommand{\curl}{\operatorname{curl}}
\newcommand{\diver}{\operatorname{div}}
\newcommand{\dist}{\operatorname{dist}}
\newcommand{\eps}{\varepsilon}
\newcommand{\dd}{\,\mathrm d}
\newcommand{\Nbh}{\mathcal N}
\newcommand{\cX}{\mathcal X}
\newcommand{\cG}{\mathcal G}
\newcommand{\cU}{\mathcal U}

\title[Compactly supported strong ill-posedness for 2D Euler]
{Compactly Supported Strong Ill-Posedness of the
Two-Dimensional Euler Flow in $W^{3,1}(\mathbb R^2)$}

\author{Yingzhe Ban, Maotuo Guo, Qi Zhang}
\date{September 2026}

\subjclass[2020]{35Q31, 35B30, 35B35, 42B15, 76B03}
\keywords{two-dimensional Euler equations, endpoint Sobolev space,
strong ill-posedness, compact support}

\begin{document}

\begin{abstract}
We prove strong ill-posedness for the two-dimensional incompressible Euler
equation in the endpoint velocity space
$W^{3,1}_\sigma(\mathbb R^2)$.  Given any compactly supported smooth
divergence-free background and any $\delta>0$, we construct a single
compactly supported divergence-free datum within $\delta$ of the background
in $W^{3,1}$.  The perturbation can be localized inside any prescribed
nonempty open neighborhood of the background support while remaining
disjoint from that support.
The associated unique global finite-energy solution $u$ satisfies
$\curl u\in C([0,\infty);W^{2,1}(\mathbb R^2))$, but
\[
 \operatorname*{ess\,sup}_{0<t<T}\|u(t)\|_{W^{3,1}}=\infty
 \qquad\text{for every }T>0.
\]
The proof combines localized hyperbolic insertion, the Bonami--Poornima
nonmultiplier mechanism, and many-packet dilution.  Moment cancellation and
positive material separation control the nonlocal Biot--Savart interactions,
allowing a nested construction inside one fixed compact region.  Without the
compact-support requirement on the limiting datum, a complementary
far-field argument yields a dense $G_\delta$ set of data whose solutions
leave $W^{3,1}$ at every positive rational time and on a residual set of
times, with unbounded essential supremum on every nonempty open time
interval.
\end{abstract}

\maketitle

\section{Introduction and main results}
\label{sec:intro}

\subsection{The endpoint velocity problem}

We consider the incompressible Euler equations on the full plane:
\begin{equation*}
 \partial_tu+u\cdot\nabla u+\nabla p=0,
 \qquad \diver u=0,
 \qquad u|_{t=0}=u_0.
\end{equation*}
For smooth divergence-free data with sufficient decay, the corresponding
solution is smooth for all time in two dimensions.  We denote by
$\cS_T^\infty u_0=u(T)$ the time-$T$ classical solution map.

With the conventions fixed in Section~\ref{sec:functional}, the vorticity
$\omega=\curl u$ obeys
\begin{equation*}
 \partial_t\omega+u\cdot\nabla\omega=0,
 \qquad
 u=\nabla^\perp(-\Delta)^{-1}\omega=K*\omega.
\end{equation*}
At the level of derivatives, the implication
\[
 u\in W^{3,1}(\R^2)\quad\Longrightarrow\quad
 \omega\in W^{2,1}(\R^2)
\]
is immediate.  The converse is false as a general elliptic estimate: the
highest derivatives of the velocity contain second-order Riesz transforms
of $D^2\omega$, and Calder\'on--Zygmund operators are not bounded on
$L^1$.  Consequently, persistence of $W^{2,1}$ vorticity does not by itself
imply persistence of $W^{3,1}$ velocity.

The space $W^{3,1}$ is the $p=1$ endpoint of the borderline velocity scale
$W^{1+2/p,p}$ that appears in the critical well-posedness theory.  For
$1<p<\infty$, the Hodge reconstruction is governed by Calder\'on--Zygmund
operators bounded on $L^p$.  At $p=1$ this basic mechanism fails, so the
endpoint cannot be obtained by formally passing to the limit in the
$p>1$ theory.

This distinction is especially relevant because recent work establishes
first short-time and then global propagation of $W^{2,1}(\R^2)$ vorticity
in two dimensions \cite{CozziHarrison2025,DanchinJeong2026}.  Related
persistence results are known in the endpoint Triebel--Lizorkin scale
$F^3_{1,\infty}$ \cite{HwangPak2025}.  These positive results leave open the
behavior of the endpoint Hodge reconstruction required for
$W^{3,1}$ velocity.

Compact support on the full plane introduces a second difficulty.  Compactly
supported vorticity generally generates a velocity with algebraic tails,
whereas constructing a compactly supported velocity from a compact stream
function forces additional moment cancellations on its vorticity.  Moreover,
the Biot--Savart law couples spatially separated regions instantaneously.  A
localized instability mechanism must therefore control both the far field
of the inserted perturbation and the response induced on the background
vorticity region.

The two-dimensional global theory originates in the work of Wolibner and
Yudovich and is developed in the geometric, Sobolev, and vorticity
formulations of
\cite{Wolibner1933,Yudovich1963,Kato1972,EbinMarsden1970,
MajdaBertozzi,Chemin1998}.  Relevant continuation and endpoint regularity
criteria include
\cite{BealeKatoMajda1984,BourguignonBrezis1974,Chae2003,
Chae2004,PakPark2004,Vishik1998,CheskidovShvydkoy2010}.  The regularity of
the Eulerian solution map is subtle even in classically well-posed Sobolev
regimes \cite{Inci2015}.

Endpoint and borderline ill-posedness mechanisms for Euler have been
developed in critical Sobolev, Besov, H\"older, and related scales; see,
among others,
\cite{BourgainLi2D2015,BourgainLi3D2015,BourgainLiCm2015,
ElgindiJeong2017,ElgindiMasmoudi2020,KimJeong2022,
MisiolekYoneda2016,Jeong2021,Kwon2021,CozziHarrisonRadke2025}.
Two levels of instability are relevant here.  Fixed-time norm inflation for
smooth data shows that the classical solution map is not locally bounded in
the endpoint topology.  Strong ill-posedness for a single datum requires an
additional condensation argument: one must pass from a sequence of finite
smooth amplifiers to one $W^{3,1}$ datum while retaining loss of uniform
control on every positive time interval.

For $1<p<\infty$, the Bourgain--Li strong formulation produces critical
velocity data arbitrarily close to smooth backgrounds for which no
$L^\infty_tW^{1+2/p,p}_x$ evolution is possible on any positive time
interval \cite{BourgainLi2D2015}.  Theorem~\ref{thm:compact-main} provides a
$p=1$ analogue at the velocity endpoint
$W^{1+2/p,p}=W^{3,1}$, in the sense of failure of
$L^\infty_tW^{3,1}_x$ control, while keeping the datum and perturbation
compactly supported.  A complementary Baire-category result gives a dense
$G_\delta$ set of bad data when compact support of the limiting datum is not
required.  Unlike mechanisms based on growth of the Lipschitz norm, the
present construction exploits the failure of anisotropically conjugated
Hodge reconstruction at the integer $L^1$ endpoint.

\subsection{Main compact-support results}

Set
\begin{equation*}
 \cX:=W^{3,1}_\sigma(\R^2)
 =\{u\in W^{3,1}(\R^2;\R^2):\diver u=0\}.
\end{equation*}
This is a closed subspace of $W^{3,1}$ and hence a Banach space.  Whenever a
velocity is not in $W^{3,1}$, its $W^{3,1}$ norm is understood to be
$+\infty$.  By Proposition~\ref{prop:endpoint-flow}, every $u_0\in\cX$
generates a unique global finite-energy solution whose vorticity belongs to
$C([0,\infty);W^{2,1})$; we denote it by $\cS_tu_0=u(t)$.

Our principal result keeps the limiting datum compactly supported and gives
arbitrarily short-time strong inflation near an arbitrary compactly
supported smooth datum.

\begin{theorem}[Compactly supported short-time strong ill-posedness]
\label{thm:compact-main}
Let $\bar u_0\in C_{c,\sigma}^\infty(\R^2)$, let $\delta>0$, and let
$\mathcal O\subset\R^2$ be any nonempty open neighborhood of
$S:=\supp\bar u_0$.  Let also $(\tau_n)$ and $(M_n)$ be arbitrary positive
sequences with $\tau_n\downarrow0$ and $M_n\uparrow\infty$.  There exist a
compactly supported datum $u_0\in\cX$ and a decreasing sequence
$t_n\downarrow0$ such that
\begin{align}
 &\supp(u_0-\bar u_0)\Subset\mathcal O\setminus S,
 \qquad \supp u_0\Subset\mathcal O,
 \qquad \|u_0-\bar u_0\|_{W^{3,1}}<\delta,
 \label{eq:compact-main-data}\\
 &0<t_n<\tau_n,
 \qquad \|\cS_{t_n}u_0\|_{W^{3,1}}>M_n.
 \label{eq:compact-main-thresholds}
\end{align}
Moreover, for every $T>0$,
\begin{equation}\label{eq:compact-main-essential}
 \operatorname*{ess\,sup}_{0<t<T}
 \|\cS_tu_0\|_{W^{3,1}}=\infty.
\end{equation}
\end{theorem}

The lower-semicontinuity argument in
Section~\ref{sec:constructive-condensation} upgrades the threshold sequence
\eqref{eq:compact-main-thresholds} to the essential-supremum conclusion
\eqref{eq:compact-main-essential}.  The upper bounds $\tau_n$ are arbitrary,
whereas the activation times $t_n$ are selected adaptively.  Prescribed-time
and residual-time conclusions are obtained in Theorem~\ref{thm:strong}
without requiring compact support of the limiting datum.  Since
$C_{c,\sigma}^\infty$ is dense in $\cX$, the compactly supported data
satisfying \eqref{eq:compact-main-essential} are themselves dense in $\cX$.

The constructive core is the following smooth fixed-time theorem.  The
homogeneous conclusion is recorded because it is the form used in the
nested-ball diagonal construction.

\begin{theorem}[Strong fixed-time norm inflation at zero]
\label{thm:zero}
For every $T>0$, every $\delta>0$, and every $M>0$, there exists
$u_0\in C_c^\infty(\R^2;\R^2)$ with $\diver u_0=0$ such that
\begin{equation*}
 \|u_0\|_{W^{3,1}(\R^2)}<\delta,
\end{equation*}
but the corresponding classical Euler solution satisfies
\begin{equation*}
 \|\cS_T^\infty u_0\|_{\dot W^{3,1}(\R^2)}>M.
\end{equation*}
In particular, its full $W^{3,1}$ norm is also larger than $M$, and
\begin{equation}\label{eq:local-unbounded-zero}
 \sup_{\substack{u_0\in C_c^\infty(\R^2;\R^2),\ \diver u_0=0\\
                  \|u_0\|_{W^{3,1}}<\delta}}
 \|\cS_T^\infty u_0\|_{\dot W^{3,1}}=\infty
 \qquad(T,\delta>0).
\end{equation}
\end{theorem}

\subsection{The endpoint mechanism and successive reductions}

Let $A\in SL(2,\R)$ be an area-preserving affine deformation.  If a
vorticity $-\Delta f$ is transported by $A$ and the velocity is then
reconstructed by the Biot--Savart law, the stream function pulled back to
the initial coordinates is acted on by
\begin{equation*}
 m_A(\xi)=\frac{|\xi|^2}{|A^{-T}\xi|^2}.
\end{equation*}
When $A$ is nonconformal, $m_A$ is a nonconstant homogeneous function of
degree zero.  Bonami and Poornima proved that such a symbol cannot be a
Fourier multiplier on any integer homogeneous Sobolev space
$\dot W^{k,1}$ \cite{BonamiPoornima1987}; see also the later multiplier
rigidity results in \cite{Curca2022,Ornstein1962,Stein1970}.  At $k=4$ this
creates arbitrarily large passive velocity amplification in
$\dot W^{3,1}$.

The passage from the full Euler equation to this passive affine response is
made through a sequence of exact decompositions followed by two controlled
approximations.  First, writing $U=\bar u+v$ gives the exact perturbation
equation around a smooth background.  The perturbation vorticity is then
split exactly into transported packet components and the response generated
on the background-vorticity region, and each packet equation is written in
moving rescaled coordinates.  In those coordinates, replacing the
background by its affine Taylor part and discarding the cross-packet and
external fields produces the nonlinear affine profile; the resulting
coefficient errors tend to zero with the packet scale $r$.  Finally, setting
the self-interaction coefficient $a$ to zero produces the passive profile,
with an error controlled by $a$.  Thus the affine and passive equations are
not uncontrolled linearizations of the original dynamics: the two passages
are quantified separately by $r\to0$ and $a\to0$.

An unbounded linearized response is not, by itself, enough to prove
nonlinear discontinuity.  The central nonlinear idea is to split a small
total perturbation into many copies.  If $J$ is the number of packets and
$a$ the amplitude of one packet, we impose
\begin{equation*}
 a=\frac{\eps}{J},
 \qquad Ja=\eps.
\end{equation*}
For the critical velocity packet used below, the top-order initial cost of
one copy is comparable to $a$, hence the total initial cost is
$Ja=\eps$.  In its own moving microscopic coordinates, however, the
nonlinear self-advection is multiplied by $a$.  Therefore increasing $J$
makes each copy perturbative without decreasing the total amplified output,
which is comparable to
\[
 J a Q=\eps Q,
\]
where $Q$ is the passive endpoint gain of one fixed profile.

The full-space geometry creates two additional issues absent from the affine
model.  First, packet centers move along trajectories of a
time-dependent compactly supported background, and the background has a
Taylor remainder.  Second, the Biot--Savart law has a long range, so one
must control both low frequencies and the feedback generated on the
background-vorticity region.  Each packet has zero circulation and physical
$L^1$ vorticity of order $ar^2$.  At a macroscopic distance its velocity is
therefore of order
\[
 r\,\|\Theta_j\|_{L^1}=O(ar^3).
\]
This is the source of the $O(r^3)$ external-response estimate.  After
rescaling by the microscopic velocity scale $ar$, the cross-packet and
external fields are respectively $O(r^2)$ and $O(a^{-1}r^2)$.  Since $a$ is
fixed before $r\to0$, the dominant coefficient error is the $O(r)$ Taylor
remainder of the background.  A high-order bootstrap then gives the
moving-packet approximation with $O(r)$ error.

\subsection{Proof architecture and parameter order}

The proof is organized as follows.
\begin{enumerate}[label=\textup{(\arabic*)},leftmargin=2.6em]
 \item Section~\ref{sec:functional} records the endpoint and localized
 Biot--Savart estimates, then constructs a compactly supported background
 with a material zero-vorticity region and nonconformal deformation.
 Section~\ref{sec:affine} converts that geometry into an endpoint affine
 amplifier and controls its small-amplitude nonlinear perturbation.

 \item Section~\ref{sec:packets} embeds finitely many copies of the affine
 profile into the Euler flow and proves the packet realization proposition
 by combining the macroscopic $O(r^3)$ response estimate with the $O(r)$
 moving-coordinate approximation.  Section~\ref{sec:local-proof} then sums
 the disjoint $L^1$ outputs and uses amplitude--time scaling to obtain
 fixed-time norm inflation at zero.

 \item Section~\ref{sec:compact-insertion} inserts a fresh finite amplifier
 into an unused ball of an arbitrary compactly supported smooth background.
 The global endpoint flow, lower semicontinuity, and a nested-ball argument
 inside one fixed compact region then prove
 Theorem~\ref{thm:compact-main}.

 \item Section~\ref{sec:far-superposition} proves the complementary
 prescribed-time result by far-field superposition and a rational-time
 diagonal construction.
\end{enumerate}

The order of parameter choices in the packet construction is equally
important.  For a target triple $(\tau,\delta,M)$, it is
\begin{equation*}
 \eps\longrightarrow Q\longrightarrow (f,R,h)
 \longrightarrow a_*\longrightarrow J
 \longrightarrow a=\eps/J
 \longrightarrow (q_1,\ldots,q_J)
 \longrightarrow r.
\end{equation*}
The profile $f$ is fixed before the amplitude is made small, and no estimate
uniform over the bad profiles is required.  Once $f$ is fixed, $J$ is chosen
to force $a<a_*$.  Only after $J$ and all labels have been fixed is the
packet scale $r$ sent to zero.  This order is recorded quantitatively in
Section~\ref{sec:parameters}.

\section{Preliminaries and background geometry}
\label{sec:functional}

\subsection{Conventions and elementary estimates}

We use
\begin{equation*}
 \nabla^\perp=(\partial_2,-\partial_1),
 \qquad
 \curl u=\partial_1u_2-\partial_2u_1.
\end{equation*}
Thus
\[
 \curl(\nabla^\perp\psi)=-\Delta\psi.
\]
The full-space Biot--Savart kernel is
\begin{equation*}
 K(x)=\frac1{2\pi}\frac{(-x_2,x_1)}{|x|^2},
 \qquad
 u=K*\omega=\nabla^\perp(-\Delta)^{-1}\omega.
\end{equation*}
For an integer $m\ge0$, set
\[
 \|F\|_{\dot W^{m,1}(\Omega)}
 =\sum_{|\alpha|=m}\|\partial^\alpha F\|_{L^1(\Omega)}.
\]
All tensor norms used for $D^mF$ are fixed once and for all.  Since these
are finite-dimensional, changing the tensor norm only changes constants.

For a set $E\subset\R^2$ and $\rho>0$, let
\[
 \Nbh_\rho(E)=\{x\in\R^2:\dist(x,E)<\rho\}.
\]
When $X(t)$ is a trajectory, the notation $X(t)+rB_R$ means
$\{X(t)+ry:y\in B_R\}$.

The highest-order seminorm is the part amplified by the construction, but
on the full plane the low-order terms of $W^{3,1}$ require decay at
infinity.  The following moment lemma makes this point explicit.

\begin{lemma}[Moment cancellation and velocity decay]
\label{lem:moments}
Let $\omega\in C_c^\infty(\R^2)$ satisfy
\begin{equation}\label{eq:moment-zero}
 \int_{\R^2}\omega(x)\,\dd x=0,
 \qquad
 \int_{\R^2}x\,\omega(x)\,\dd x=0.
\end{equation}
Then $u=K*\omega$ belongs to $W^{m,1}(\R^2)$ for every integer $m\ge0$.
More precisely, if $\supp\omega\subset B_R$, then for $|x|\ge2R$,
\begin{equation}\label{eq:far-decay}
 |D^mu(x)|
 \le C_m|x|^{-m-3}
 \int_{B_R}|y|^2|\omega(y)|\,\dd y.
\end{equation}
If $\omega(t)$ is a smooth compactly supported two-dimensional Euler
vorticity and \eqref{eq:moment-zero} holds at $t=0$, then both cancellations
hold for all time.  In particular, if $\psi\in C_c^\infty(\R^2)$ and
$u_0=\nabla^\perp\psi$, then the corresponding classical Euler velocity
belongs to $W^{m,1}(\R^2)$ for every finite time and every integer $m\ge0$.
\end{lemma}

\begin{proof}
For $|x|\ge2R$, Taylor's formula gives
\begin{align*}
 D^mK(x-y)
 &=D^mK(x)-D^{m+1}K(x)[y]
   +\int_0^1(1-\theta)D^{m+2}K(x-\theta y)[y,y] \,\dd\theta,
\end{align*}
The constant and linear terms disappear after integration against $\omega$
by \eqref{eq:moment-zero}.  Moreover,
\[
 |D^{m+2}K(z)|\le C_m|z|^{-m-3},
 \qquad |x-\theta y|\ge\frac{|x|}{2}.
\]
This proves \eqref{eq:far-decay}.  On a fixed bounded
set, standard elliptic regularity and smoothness of $\omega$ imply local
integrability of every derivative of $u$.  The far-field estimate then gives
$u\in W^{m,1}$.

The circulation is conserved by the vorticity transport equation.  For the
first moment, integration by parts gives
\begin{equation*}
 \frac{\dd}{\dd t}\int x\omega(t,x)\,\dd x
 =\int u(t,x)\omega(t,x)\,\dd x.
\end{equation*}
Using $u=K*\omega$ and $K(-z)=-K(z)$,
\[
 \int u\omega
 =\iint K(x-y)\omega(x)\omega(y)\,\dd x\dd y=0
\]
by interchanging $x$ and $y$.  This proves preservation of the first
moment.  Finally, if $\omega_0=-\Delta\psi$ with
$\psi\in C_c^\infty$, integration by parts gives both identities in
\eqref{eq:moment-zero}.  The Euler vorticity remains smooth and compactly
supported on finite time intervals, so the first part of the lemma applies
at every time.
\end{proof}

\begin{remark}[Moments of individual packets]
For the macroscopic far-field estimate of an individual packet we use only
its zero circulation.  Its first moment need not be conserved relative to a
fixed packet center.  The stronger zero mass and zero first moment are used
for the \emph{total} Euler vorticities generated by compact stream functions,
which guarantees that the full velocities lie in $W^{3,1}$.
\end{remark}

For later comparisons of perturbed and background trajectories, we also
record the following elementary flow estimate.

\begin{lemma}[Comparison of two trajectories]
\label{lem:flow-comparison}
Let $b$ and $b+z$ be smooth vector fields on $[0,T]$, with
$\|\nabla b\|_{L_t^\infty L_x^\infty}\le L$.  Let $X$ and $Y$ solve
\[
 \dot X=b(t,X),\qquad \dot Y=b(t,Y)+z(t,Y).
\]
Then
\begin{equation*}
 |Y(t)-X(t)|
 \le e^{Lt}\left(|Y(0)-X(0)|+
 \int_0^t\|z(s)\|_{L^\infty}\,\dd s\right).
\end{equation*}
The estimate uses no derivative of $z$.
\end{lemma}

\begin{proof}
Subtract the two ODEs, use the Lipschitz bound for $b$, and apply
Gronwall's inequality.
\end{proof}

\subsection{Biot--Savart estimates}
\label{sec:tools}

The following proposition collects the estimates used in the affine-profile,
packet, and localized-insertion arguments, with their distinct cancellation
hypotheses stated explicitly.

\begin{proposition}[Localized Biot--Savart estimates]
\label{prop:BS-toolbox}
The following estimates hold.
\begin{enumerate}[label=\textup{(\roman*)},leftmargin=2.6em]
 \item If $g\in L^1(\R^2)\cap L^\infty(\R^2)$, then
 \begin{equation*}
  \|K*g\|_{L^\infty}
  \le C\|g\|_{L^1}^{1/2}\|g\|_{L^\infty}^{1/2}.
 \end{equation*}

 \item Let $E,F\subset\R^2$ be compact with $\dist(E,F)\ge d>0$.  If
 $g\in L^1$ is supported in $F$, then for every integer $m\ge0$,
 \begin{equation*}
  \|D^mK*g\|_{L^\infty(E)}
  \le C_{m,d}\|g\|_{L^1}.
 \end{equation*}
 If in addition $\int g=0$, $\supp g\subset B(X,\varrho)$, and
 $\dist(E,B(X,\varrho))\ge d$, then cancellation improves this to
 \begin{equation*}
  \|D^mK*g\|_{L^\infty(E)}
  \le C_{m,d}\,\varrho\|g\|_{L^1}.
 \end{equation*}

 \item Let $s\ge0$ be an integer.  If
 $g\in H^s(\R^2)\cap L^1(\R^2)$ has zero integral and is supported in
 $B_{R_*}$, then
 \begin{equation}\label{eq:low-frequency}
  \|K*g\|_{H^{s+1}}
  \le C_{s,R_*}\bigl(\|g\|_{H^s}+\|g\|_{L^1}\bigr).
 \end{equation}
 If $g_n\to g$ in $H^s\cap L^1$ with common support and zero mass, then
 $K*g_n\to K*g$ in $H^{s+1}$.

\end{enumerate}
\end{proposition}

\begin{proof}
For \textup{(i)}, split the convolution at radius $\ell$:
\[
 |K*g(x)|\le C\ell\|g\|_{L^\infty}
 +C\ell^{-1}\|g\|_{L^1},
\]
and optimize in $\ell$.  The first estimate in \textup{(ii)} follows
because all derivatives of $K(x-y)$ are uniformly bounded on $E\times F$.
For the improved estimate, use zero mass to write
\[
 D^mK*g(x)=\int
 \bigl[D^mK(x-y)-D^mK(x-X)\bigr]g(y)\,\dd y
\]
and apply the mean value theorem.

For \textup{(iii)}, the Fourier identity
\[
 \widehat{K*g}(\xi)=i\frac{\xi^\perp}{|\xi|^2}\widehat g(\xi)
\]
controls high frequencies directly.  At $|\xi|<1$, zero mass and the
support bound give
\[
 |\widehat g(\xi)|
 =\left|\int(e^{-ix\cdot\xi}-1)g(x)\,\dd x\right|
 \le R_*|\xi|\|g\|_{L^1}.
\]
This proves \eqref{eq:low-frequency} and its convergence assertion.
\end{proof}

\subsection{A compactly supported hyperbolic background}
\label{sec:background}

Choose $\rho>0$ and a radial cutoff $\chi\in C_c^\infty(\R^2)$ such that
\begin{equation*}
 \chi=1\quad\text{on }B_{4\rho},
 \qquad
 \supp\chi\subset B_{8\rho}.
\end{equation*}
Define
\begin{equation*}
 \bar\psi_0(x)=\chi(x)x_1x_2,
 \qquad
 \bar u_0=\nabla^\perp\bar\psi_0,
 \qquad
 \bar\omega_0=-\Delta\bar\psi_0.
\end{equation*}
Then $\bar u_0\in C_c^\infty$ and, on $B_{4\rho}$,
\begin{equation*}
 \bar u_0(x)=(x_1,-x_2)=Mx,
 \qquad
 \bar\omega_0(x)=0,
 \qquad
 M=\begin{pmatrix}1&0\\0&-1\end{pmatrix}.
\end{equation*}
Let $\bar u(t)=\cS_t^\infty\bar u_0$, let $\bar\omega=\curl\bar u$, and
let $\bar\Phi(t,s,x)$ be the background flow.  For a label $q$, set
\begin{equation*}
 X_q(t)=\bar\Phi(t,0,q),
 \qquad
 B_q(t)=D_q\bar\Phi(t,0,q),
 \qquad
 M_q(t)=\nabla\bar u(t,X_q(t)).
\end{equation*}
Then
\begin{equation*}
 \dot B_q(t)=M_q(t)B_q(t),
 \qquad B_q(0)=I,
 \qquad \det B_q(t)=1.
\end{equation*}

Let
\begin{equation*}
 \cP_t=\bar\Phi(t,0,\overline{B_{2\rho}}),
 \qquad
 \cZ_t=\bar\Phi(t,0,\supp\nabla\bar\omega_0).
\end{equation*}
The notation $\cZ_t$ is convenient because
\begin{equation*}
 \supp\nabla\bar\omega(t)=\cZ_t.
\end{equation*}
Indeed,
\[
 \nabla\bar\omega(t,\bar\Phi(t,0,q))
 =B_q(t)^{-T}\nabla\bar\omega_0(q).
\]

\begin{proposition}[Material zero-vorticity tube and nonconformality]
\label{prop:background}
There exist $\tau_0>0$ and $d_0>0$ such that for every
$0\le t\le\tau_0$:
\begin{enumerate}[label=\textup{(\roman*)},leftmargin=2.4em]
 \item $\bar\omega(t)=0$ on $\bar\Phi(t,0,B_{4\rho})$;
 \item
 \begin{equation}\label{eq:material-separation}
  \dist(\cP_t,\cZ_t)\ge d_0;
 \end{equation}
 \item for every $0<t\le\tau_0$, $B_0(t)$ is nonconformal;
 \item the maps $(t,q)\mapsto X_q(t),B_q(t),M_q(t)$ and all derivatives
 needed below are uniformly bounded for
 $(t,q)\in[0,\tau_0]\times\overline{B_{2\rho}}$.
\end{enumerate}
After reducing $\tau_0$ if necessary, one may also require
\begin{equation*}
 C_\Phi^{-1}|q-q'|
 \le |X_q(t)-X_{q'}(t)|
 \le C_\Phi|q-q'|
\end{equation*}
for a fixed $C_\Phi\ge1$ and all
$q,q'\in\overline{B_{2\rho}}$.
\end{proposition}

\begin{proof}
The vorticity is transported by the background flow:
\[
 \bar\omega(t,\bar\Phi(t,0,q))=\bar\omega_0(q).
\]
Since $\bar\omega_0=0$ on $B_{4\rho}$, this proves (i).  At time zero the
compact sets $\overline{B_{2\rho}}$ and
$\supp\nabla\bar\omega_0$ have positive distance.  Their images under the
smooth diffeomorphism $\bar\Phi(t,0,\cdot)$ remain disjoint.  Compactness in
$t$ and continuity yield the uniform lower bound \eqref{eq:material-separation}.

At $q=0$, $B_0(0)=I$ and $\dot B_0(0)=M$.  Therefore
\begin{equation*}
 B_0(t)^TB_0(t)=I+2tM+O(t^2).
\end{equation*}
The two diagonal entries differ by $4t+O(t^2)$, so the matrix is not a
scalar multiple of the identity for all sufficiently small $t>0$.  This is
precisely nonconformality in dimension two.  The remaining assertions follow
from smooth classical well-posedness, compactness, and the inverse function
theorem for the flow.
\end{proof}

For later support bootstraps, fix
\begin{equation*}
 0<\kappa<\frac{d_0}{32C_\Phi^2}.
\end{equation*}
Then, after reducing $\kappa$ if needed,
\begin{equation*}
 \dist\bigl(\Nbh_{4\kappa}(\cP_t),
             \Nbh_{4\kappa}(\cZ_t)\bigr)\ge\frac{d_0}{2}
 \qquad(0\le t\le\tau_0).
\end{equation*}
The packet and external-response tubes will remain in these fixed
neighborhoods of the background material regions.

\section{The endpoint affine amplifier}
\label{sec:affine}

\subsection{The nonmultiplier theorem}

We recall the endpoint rigidity theorem used in the construction.

\begin{theorem}[Bonami--Poornima]
\label{thm:BP}
Let $d\ge2$, let $k\ge1$ be an integer, and let
$m\in C(\R^d\setminus\{0\})$ be homogeneous of degree zero.  If $m(D)$ is
bounded on $\dot W^{k,1}(\R^d)$, then $m$ is constant.
\end{theorem}

For $A\in SL(2,\R)$, set
\begin{equation*}
 m_A(\xi)=\frac{|\xi|^2}{|A^{-T}\xi|^2},
 \qquad \xi\ne0.
\end{equation*}
The symbol $m_A$ is positive, smooth away from the origin, and homogeneous
of degree zero.

If $m_A$ is constant, then
$A^{-1}A^{-T}=cI$; the determinant condition forces $c=1$, so $A$ is
orthogonal.  Thus $m_A$ is nonconstant exactly when $A$ is nonconformal.
The only additional point needed below is that the resulting multiplier
unboundedness can be witnessed by compactly supported smooth functions.
The cutoff argument establishing this fact is recorded as
Lemma~\ref{lem:compact-reduction} in Appendix~\ref{app:compact-reduction}.

\begin{proposition}[Bad compactly supported affine profiles]
\label{prop:bad-profile}
Let $A\in SL(2,\R)$ be nonconformal.  For every $L>0$, there exists a
real-valued $f\in C_c^\infty(\R^2)$ such that
\begin{equation}\label{eq:f-normal}
 \|f\|_{\dot W^{4,1}}\le1,
 \qquad
 \|m_A(D)f\|_{\dot W^{4,1}}>L.
\end{equation}
For every such compactly supported $f$, the quantity on the right is finite.
\end{proposition}

\begin{proof}
If the ratio in \eqref{eq:f-normal} were uniformly bounded on
$C_c^\infty$, Lemma~\ref{lem:compact-reduction} would make $m_A$ a
multiplier on $\dot W^{4,1}$.  Since nonconformality makes $m_A$
nonconstant, this contradicts Theorem~\ref{thm:BP}.  Taking a real or
imaginary part and normalizing gives the first assertion.

To see finiteness for a fixed $f$, note that the vorticity
$(-\Delta f)(A^{-1}\cdot)$ is compactly supported and has zero mass and zero
first moment.  Its Biot--Savart velocity is therefore in $W^{3,1}$ by
Lemma~\ref{lem:moments}.  Lemma~\ref{lem:affine-identity} below identifies
its top derivatives with those of $m_A(D)f$ up to a fixed invertible tensor
map.
\end{proof}

\subsection{Passive affine profiles}

Fix $0<\tau\le\tau_0$.  For $q\in B_{2\rho}$, consider
\begin{equation}\label{eq:passive-q}
 \partial_t\zeta_q^0+M_q(t)y\cdot\nabla\zeta_q^0=0,
 \qquad
 \zeta_q^0(0,y)=-\Delta f(y),
 \qquad
 w_q^0=K*\zeta_q^0.
\end{equation}
Because $\operatorname{tr}M_q=0$, the linear drift is divergence-free.  Its
flow is $y=B_q(t)y_0$, and hence
\begin{equation*}
 \zeta_q^0(t,y)=(-\Delta f)(B_q(t)^{-1}y).
\end{equation*}

\begin{lemma}[Affine multiplier identity]
\label{lem:affine-identity}
Let $\psi_q^0$ be the stream function of $w_q^0$ and define
$\phi_{q,t}(z)=\psi_q^0(t,B_q(t)z)$.  Then
\begin{equation}\label{eq:phi-multiplier}
 \phi_{q,t}=m_{B_q(t)}(D)f.
\end{equation}
There exist constants $0<c(B)\le C(B)<\infty$, locally uniform for
$B\in SL(2,\R)$, such that
\begin{equation}\label{eq:norm-equivalence}
 c(B_q(t))\|m_{B_q(t)}(D)f\|_{\dot W^{4,1}}
 \le \|w_q^0(t)\|_{\dot W^{3,1}}
 \le C(B_q(t))\|m_{B_q(t)}(D)f\|_{\dot W^{4,1}}.
\end{equation}
\end{lemma}

\begin{proof}
The equation $-\Delta_y\psi_q^0=\zeta_q^0$, the change of variables
$y=B_q(t)z$, and $\det B_q(t)=1$ yield
\begin{equation*}
 -\diver_z\bigl(B_q(t)^{-1}B_q(t)^{-T}\nabla_z\phi_{q,t}\bigr)
 =-\Delta_zf.
\end{equation*}
Taking Fourier transforms gives
\[
 |B_q(t)^{-T}\xi|^2\widehat{\phi}_{q,t}(\xi)
 =|\xi|^2\widehat f(\xi),
\]
which is \eqref{eq:phi-multiplier}.  Moreover,
\begin{equation*}
 w_q^0(t,B_q(t)z)
 =\mathsf J B_q(t)^{-T}\nabla_z\phi_{q,t}(z),
 \qquad
 \mathsf J=\begin{pmatrix}0&1\\-1&0\end{pmatrix}.
\end{equation*}
The third derivatives on the left and fourth derivatives of $\phi_{q,t}$
are related by a fixed invertible linear map on finite-dimensional tensors.
The change of variables preserves measure.  This proves
\eqref{eq:norm-equivalence} and its local uniformity in $B$.
\end{proof}

\begin{proposition}[Uniform passive amplification near a nonconformal label]
\label{prop:uniform-passive}
Let $\mathcal Q\subset\R^2$ be open, let $q_0\in\mathcal Q$, and let
$q\mapsto A_q$ be continuous from $\mathcal Q$ to $SL(2,\R)$ with
$A_{q_0}$ nonconformal.  For every $Q>0$, there exist a real-valued
$f\in C_c^\infty(\R^2)$, a radius $R>0$, and a neighborhood
$V\Subset\mathcal Q$ of $q_0$ such that, writing
$w_q=K*((-\Delta f)(A_q^{-1}\cdot))$,
\begin{equation}\label{eq:uniform-passive}
 \|f\|_{\dot W^{4,1}}\le1,
 \qquad
 \int_{B_R}|D^3w_q(y)|\,\dd y>Q
 \quad\text{for every }q\in V.
\end{equation}
\end{proposition}

\begin{proof}
Proposition~\ref{prop:bad-profile} and
the calculation in Lemma~\ref{lem:affine-identity}, applied with
$B_q(t)$ replaced by $A_{q_0}$, give $f$ with
$\|w_{q_0}\|_{\dot W^{3,1}}>4Q$.  Since
$D^3w_{q_0}\in L^1(\R^2)$, choose $R$ such that
\[
 \int_{B_R}|D^3w_{q_0}(y)|\,\dd y>2Q.
\]
For $A$ in a compact subset of $SL(2,\R)$, the sources
$(-\Delta f)(A^{-1}\cdot)$ have a common compact support and depend
continuously on $A$ in $C_c^\infty$.  Moving three derivatives onto the
source in the Biot--Savart convolution and using $K\in L^1_{\rm loc}$
shows that
$A\mapsto K*((-\Delta f)(A^{-1}\cdot))$ is continuous into
$W^{3,1}(B_R)$.  Continuity of $q\mapsto A_q$ now gives the required
neighborhood $V$.
\end{proof}

Fix $0<\tau\le\tau_0$ and $Q>0$.  By
Proposition~\ref{prop:background}, $B_0(\tau)$ is nonconformal.  Apply
Proposition~\ref{prop:uniform-passive} to $A_q=B_q(\tau)$ and shrink its
label neighborhood to a ball $B_h$, $0<h<\rho$.  Then the profile
\eqref{eq:passive-q} satisfies \eqref{eq:uniform-passive} with
$w_q=w_q^0(\tau)$ for every $q\in B_h$.

\begin{remark}[Why the observation ball need not contain the vorticity]
The $L^1$ mass of $D^3w_q^0$ producing the large output need not be
concentrated on the support of $\zeta_q^0$.  Later rescaled estimates are
therefore carried out on a ball strictly larger than both the vorticity
support and the observation ball $B_R$.
\end{remark}

\subsection{Nonlinear affine profiles}
\label{sec:profiles}

For $q\in B_h$ and $a\ge0$, consider the nonlinear profile system
\begin{equation}\label{eq:nonlinear-profile}
 \left\{
 \begin{aligned}
 &\partial_t\zeta_q^a+
 \bigl(M_q(t)y+a w_q^a\bigr)\cdot\nabla\zeta_q^a=0,\\
 &w_q^a=K*\zeta_q^a,\\
 &\zeta_q^a(0)=-\Delta f.
 \end{aligned}
 \right.
\end{equation}
The coefficient $a$ is the physical vorticity amplitude of a microscopic
packet after its amplitude has been factored out.  Since
$\operatorname{tr}M_q=0$ and $w_q^a$ is divergence-free, all $L^p$ norms of
$\zeta_q^a$ are conserved.

\begin{lemma}[Moment identities for affine profiles]
\label{lem:profile-moments}
For every $q$, $a$, and $t$,
\begin{equation*}
 \int\zeta_q^a(t,y)\,\dd y=0,
 \qquad
 \int y\,\zeta_q^a(t,y)\,\dd y=0.
\end{equation*}
Consequently, $w_q^a(t)\in W^{m,1}(\R^2)$ for every integer $m\ge0$.
\end{lemma}

\begin{proof}
The mass is conserved by the divergence-free transport.  Let
$I(t)=\int y\zeta_q^a(t,y)\,\dd y$.  Integration by parts gives
\[
 I'(t)=M_q(t)I(t)+a\int w_q^a(t,y)\zeta_q^a(t,y)\,\dd y.
\]
The last integral vanishes by oddness of the Biot--Savart kernel, exactly as
in the proof of Lemma~\ref{lem:moments}.  Since $I(0)=0$, uniqueness for the
linear ODE yields $I(t)=0$.  Compact support on bounded time intervals and
Lemma~\ref{lem:moments} give the velocity statement.
\end{proof}

\begin{lemma}[Uniform smooth profile bounds]
\label{lem:profile-smooth-bounds}
Fix $0<\tau\le\tau_0$, $h>0$, $A_0>0$, and an integer $s\ge4$.
There are $R_{\rm prof}$ and $C_{\rm prof}$, depending on
$f,\tau,h,A_0,s$ and the background, such that for every
$q\in\overline{B_h}$ and $0\le a\le A_0$,
\begin{align}
 &\supp\zeta_q^a(t)\subset B_{R_{\rm prof}}
 \qquad(0\le t\le\tau),\label{eq:profile-support}\\
 &\sup_{0\le t\le\tau}
 \|\zeta_q^a(t)\|_{H^s}\le C_{\rm prof}.
 \label{eq:profile-Hs}
\end{align}
\end{lemma}

\begin{proposition}[Uniform small-amplitude convergence]
\label{prop:small-a}
Fix the profile $f$, the radii $R,h$, and $0<\tau\le\tau_0$ furnished by
Proposition~\ref{prop:uniform-passive}.  There exist $a_*>0$ and
$C_{\rm prof}<\infty$ such that
\begin{equation}\label{eq:small-a}
 \sup_{q\in B_h}
 \|w_q^a-w_q^0\|_{C([0,\tau];W^{3,1}(B_R))}
 \le C_{\rm prof}a
 \qquad(0<a<a_*).
\end{equation}
After reducing $a_*$, one has
\begin{equation}\label{eq:nonlinear-output}
 \int_{B_R}|D^3w_q^a(\tau,y)|\,\dd y
 \ge\frac34Q
\end{equation}
for every $q\in B_h$ and $0<a<a_*$.
\end{proposition}

The proofs of Lemma~\ref{lem:profile-smooth-bounds} and
Proposition~\ref{prop:small-a}, including the uniformity in $(q,a)$, are
given in Appendix~\ref{app:profile-estimates}.  They are self-contained
a priori estimates based on characteristics, the logarithmic Biot--Savart
bound, and high-order transport inequalities; deferring them keeps the
affine amplification mechanism visible in the main text.

\begin{remark}[No uniformity over bad profiles]
The constants $a_*$ and $C_{\rm prof}$ may depend badly on $f$.  This is not
a defect: $f$ is fixed first, and only afterward is the number of packets
$J$ increased so that $a=\eps/J<a_*$.  The argument never compares the
higher Sobolev norms of a sequence of bad profiles with the corresponding
nonmultiplier gains.
\end{remark}

\section{Moving microscopic packets and their realization}
\label{sec:packets}

\subsection{Physical packet construction and realization}

Fix from now on $0<\tau\le\tau_0$ and the objects
$f,R,h,Q,a_*$ obtained above.  Choose $R_0$ such that
$\supp f\subset B_{R_0}$.  Once an integer $J$ is fixed, choose distinct
labels
\begin{equation*}
 q_1,\ldots,q_J\in B_h.
\end{equation*}
For fixed $J$, compactness and the bi-Lipschitz property of the background
flow give
\begin{equation}\label{eq:trajectory-separation}
 d_J:=\min_{j\ne k}\min_{0\le t\le\tau}
 |X_{q_j}(t)-X_{q_k}(t)|>0.
\end{equation}
For $a,r>0$, define the initial perturbation stream function and velocity by
\begin{equation}\label{eq:packet-stream}
 \Psi_0^{J,r}(x)
 =ar^2\sum_{j=1}^Jf\left(\frac{x-q_j}{r}\right),
 \qquad
 v_0^{J,r}=\nabla^\perp\Psi_0^{J,r},
 \qquad
 U_0^{J,r}=\bar u_0+v_0^{J,r}.
\end{equation}
For sufficiently small $r$, the packet supports are disjoint and lie inside
$B_{2\rho}$.

\begin{lemma}[Initial $W^{3,1}$ cost]
\label{lem:initial-cost}
There is a universal constant $C_0\ge1$, depending only on the chosen tensor
norm, such that
\begin{equation}\label{eq:initial-cost}
 \|v_0^{J,r}\|_{W^{3,1}}
 \le C_0Ja\|f\|_{\dot W^{4,1}}
 +C_fJa(r+r^2+r^3).
\end{equation}
In particular, if $Ja=\eps$ and $\|f\|_{\dot W^{4,1}}\le1$, then
\begin{equation*}
 \limsup_{r\to0}\|v_0^{J,r}\|_{W^{3,1}}\le C_0\eps.
\end{equation*}
\end{lemma}

\begin{proof}
For one packet,
\[
 v_{0,j}^{J,r}(x)
 =ar\,(\nabla^\perp f)\left(\frac{x-q_j}{r}\right).
\]
For $0\le k\le3$,
\begin{equation*}
 \|D^kv_{0,j}^{J,r}\|_{L^1}
 =ar^{3-k}\|D^k\nabla^\perp f\|_{L^1}.
\end{equation*}
At $k=3$ this is comparable to $a\|f\|_{\dot W^{4,1}}$.  The packet
supports are disjoint, so all $L^1$ norms add over $j$.
\end{proof}

Let $U^{J,r}$ be the classical Euler solution with datum $U_0^{J,r}$ and
set
\begin{equation*}
 v=U^{J,r}-\bar u,
 \qquad
 \Theta=\curl v.
\end{equation*}
The total vorticities of both $U^{J,r}$ and $\bar u$ satisfy the two moment
conditions in Lemma~\ref{lem:moments}, because both initial velocities are
generated by compactly supported stream functions.  Thus $U^{J,r}(t)$,
$\bar u(t)$, and $v(t)$ belong to $W^{3,1}$ for every finite time.

Let $\Phi(t,s,x)$ denote the flow of the full perturbed velocity
$\bar u+v$.  Decompose the perturbation vorticity by defining
$\Theta_j$ through
\begin{equation*}
 \partial_t\Theta_j+(\bar u+v)\cdot\nabla\Theta_j=0,
 \qquad
 \Theta_j(0,x)=a\zeta_0\left(\frac{x-q_j}{r}\right),
 \qquad
 \zeta_0=-\Delta f,
\end{equation*}
and set
\begin{equation*}
 \Theta_{\rm p}=\sum_{j=1}^J\Theta_j,
 \qquad
 \Theta_{\rm e}=\Theta-\Theta_{\rm p}.
\end{equation*}
Then
\begin{equation}\label{eq:theta-e}
 \partial_t\Theta_{\rm e}+(\bar u+v)\cdot\nabla\Theta_{\rm e}
 =-v\cdot\nabla\bar\omega,
 \qquad
 \Theta_{\rm e}(0)=0.
\end{equation}
The components $\Theta_j$ are transported by one common volume-preserving
flow, so their supports remain disjoint.  Since $\int\zeta_0=0$,
\begin{align}
 \int\Theta_j(t,x)\,\dd x&=0,
 \notag\\
 \sum_{j=1}^J\|\Theta_j(t)\|_{L^1}
 &=Jar^2\|\zeta_0\|_{L^1},
 \label{eq:packet-L1}\\
 \|\Theta_{\rm p}(t)\|_{L^\infty}
 &=a\|\zeta_0\|_{L^\infty}.
 \label{eq:packet-Linfty}
\end{align}
The last identity is where disjointness under the common flow is used.

The construction is reduced to the following quantitative statement, proved
in the remainder of this section.  No later part of the norm-inflation
argument uses the internal bootstrap estimates separately.

\begin{proposition}[Packet realization]
\label{prop:packet-realization}
Fix $0<\tau\le\tau_0$, a profile $f$ and radii $R,h$ furnished by
Proposition~\ref{prop:uniform-passive}, an amplitude $0<a<a_*$, a finite
integer $J$, and distinct labels $q_1,\ldots,q_J\in B_h$.  Let
$U^{J,r}=\bar u+v$ be the Euler solution generated by
\eqref{eq:packet-stream}.  There exist $C_{\rm real}<\infty$ and
$r_{\rm real}>0$, depending on these fixed objects but not on $r$, such that
for $0<r<r_{\rm real}$ the observation sets
\begin{equation}\label{eq:realization-sets}
 \Omega_j=X_{q_j}(\tau)+rB_R
\end{equation}
are pairwise disjoint and
\begin{equation}\label{eq:realization-D3}
 \max_{1\le j\le J}
 \left\|
  r^2D_x^3v(\tau,X_{q_j}(\tau)+r\,\cdot)
  -aD_y^3w_{q_j}^a(\tau)
 \right\|_{L^1(B_R)}
 \le C_{\rm real}ar.
\end{equation}
\end{proposition}

We now prove Proposition~\ref{prop:packet-realization}.  The estimates are
organized by scale and by the region in which the Biot--Savart field is
observed.

\subsection{Localized response estimate}

The external vorticity is created only where $\nabla\bar\omega$ is
nonzero.  We isolate the Sobolev estimate needed there.  Its formulation is
local in the prescribed forcing field but global in the unknown response.

\begin{lemma}[Localized response bootstrap]
\label{lem:localized-response}
Let $s\ge5$.  Let $\bar u,\bar\omega$ be fixed smooth fields on
$[0,\tau]\times\R^2$, with $\bar u$ divergence-free.  Let $g$ be a smooth
divergence-free field and let $q$ solve
\begin{equation}\label{eq:abstract-response}
 \partial_tq+(\bar u+g+K*q)\cdot\nabla q
 =-(g+K*q)\cdot\nabla\bar\omega,
 \qquad q(0)=0.
\end{equation}
Assume:
\begin{enumerate}[label=\textup{(\roman*)},leftmargin=2.4em]
 \item $\int q(t)=0$;
 \item $\supp q(t)$ is contained in a fixed compact set $\cE$;
 \item for a fixed neighborhood $\cE_1$ of
 $\cE\cup\bigcup_t\supp\nabla\bar\omega(t)$,
 \begin{equation*}
  \sup_{0\le t\le\tau}\|g(t)\|_{C^{s+1}(\cE_1)}\le\mu.
 \end{equation*}
\end{enumerate}
Then there are $\mu_0,C>0$, depending only on the fixed background,
$\cE_1,s$, and $\tau$, such that for $0<\mu\le\mu_0$,
\begin{equation*}
 \sup_{0\le t\le\tau}
 \bigl(\|q(t)\|_{H^s}+\|q(t)\|_{L^1}\bigr)
 \le C\mu.
\end{equation*}
Consequently,
\begin{equation}\label{eq:localized-velocity-bound}
 \sup_{0\le t\le\tau}
 \|K*q(t)\|_{H^{s+1}}\le C\mu.
\end{equation}
\end{lemma}

\begin{proof}
Put
\[
 X(t)=\|q(t)\|_{H^s}+\|q(t)\|_{L^1},
 \qquad h=K*q.
\]
By Proposition~\ref{prop:BS-toolbox}\textup{(iii)} and Sobolev embedding,
\begin{equation}\label{eq:h-from-X}
 \|h(t)\|_{H^{s+1}}+\|h(t)\|_{W^{1,\infty}}
 \le C X(t).
\end{equation}
Since the transport field is divergence-free, the smooth $L^1$ inequality
is
\begin{align}
 \frac{\dd}{\dd t}\|q(t)\|_{L^1}
 &\le \|(g+h)\cdot\nabla\bar\omega\|_{L^1}\notag\\
 &\le C_{\bar\omega,\cE_1}
 \left(\|g\|_{L^\infty(\cE_1)}+\|h\|_{L^\infty}\right)
 \le C\mu+CX(t).
 \label{eq:response-L1}
\end{align}

For the $H^s$ estimate, set $b=\bar u+g+h$ and
$F=-(g+h)\cdot\nabla\bar\omega$.  Applying $D^\alpha$,
$|\alpha|\le s$, and using $\diver b=0$ yields
\begin{align}
 \frac12\frac{\dd}{\dd t}\|q\|_{H^s}^2
 &\le C\left(
  1+\|g\|_{C^{s+1}(\cE_1)}+\|h\|_{H^{s+1}}
  \right)\|q\|_{H^s}^2
  +\|F\|_{H^s}\|q\|_{H^s}.
 \label{eq:response-Hs-square}
\end{align}
Only local derivatives of $g$ enter because every derivative of $q$ is
supported in $\cE$, while all factors containing derivatives of
$\bar\omega$ are supported in $\cE_1$.  The Moser product estimate and
\eqref{eq:h-from-X} give
\begin{equation*}
 \|F\|_{H^s}
 \le C\left(\mu+\|h\|_{H^{s+1}}\right)
 \le C(\mu+X).
\end{equation*}
After dividing \eqref{eq:response-Hs-square} by $\|q\|_{H^s}$ through the
usual regularization $\sqrt{\|q\|_{H^s}^2+\nu^2}$ and letting $\nu\downarrow0$,
we obtain, on the interval where $X\le1$,
\begin{equation}\label{eq:response-Hs}
 \frac{\dd}{\dd t}\|q(t)\|_{H^s}
 \le C X(t)+C X(t)^2+C\mu.
\end{equation}
Combining \eqref{eq:response-L1}--\eqref{eq:response-Hs},
\begin{equation*}
 X'(t)\le CX(t)+CX(t)^2+C\mu,
 \qquad X(0)=0.
\end{equation*}
On the maximal interval on which $X\le1$, Gronwall's inequality yields
$X(t)\le C_\tau\mu$.  Choosing $\mu_0$ so that
$C_\tau\mu_0<1/2$ closes the bootstrap on $[0,\tau]$.
Equation \eqref{eq:localized-velocity-bound} follows from
Proposition~\ref{prop:BS-toolbox}\textup{(iii)}.
\end{proof}

\begin{remark}[Why the support hypothesis is included]
The global field $g$ need not be small.  Only its behavior on the compact
region where the response and the fixed background source are supported
enters the estimate.  This is exactly the situation for the packet field:
it is $O(r)$ globally but $O(r^3)$ on the background-vorticity tube.
\end{remark}

\subsection{Macroscopic bootstrap and support propagation}
\label{sec:macro}

We now prove simultaneously that the packet vorticities stay in the
material zero-vorticity region, that the external response stays near the
background-vorticity region, and that the perturbation velocity is $O(r)$.
This section makes the support argument and all bootstrap dependencies
explicit.

Fix for the moment
\begin{equation}\label{eq:fixed-macro-parameters}
 f,\quad a>0,\quad J\in\N,\quad
 q_1,\ldots,q_J\in B_h,\quad 0<\tau\le\tau_0.
\end{equation}
Only $r$ will tend to zero.  Let
\[
 v_{\rm p}=K*\Theta_{\rm p},
 \qquad
 v_{\rm e}=K*\Theta_{\rm e},
 \qquad
 v=v_{\rm p}+v_{\rm e}.
\]
Fix an integer $s_*\ge10$.

\begin{proposition}[Macroscopic localization and external response]
\label{prop:macro}
For the fixed parameters in \eqref{eq:fixed-macro-parameters}, there exist
constants $C_{\rm mac}<\infty$, $R_*>0$, and $r_{\rm mac}>0$ such that for
$0<r<r_{\rm mac}$ and $0\le t\le\tau$,
\begin{align}
 &\|v(t)\|_{L^\infty}\le C_{\rm mac}r,
 \notag\\
 &\|\Theta_{\rm e}(t)\|_{L^1}
  +\|\Theta_{\rm e}(t)\|_{H^{s_*}}
  \le C_{\rm mac}r^3,
 \label{eq:theta-e-small}\\
 &\|v_{\rm e}(t)\|_{H^{s_*+1}}
  +\|v_{\rm e}(t)\|_{W^{1,\infty}}
  \le C_{\rm mac}r^3,
 \label{eq:ve-small}\\
 &\supp\Theta_j(t)
  \subset X_{q_j}(t)+rB_{R_*},
 \label{eq:packet-support}\\
 &\supp\Theta_{\rm e}(t)
  \subset\Nbh_{\kappa/2}(\cZ_t).
 \label{eq:external-support-final}
\end{align}
The packet supports are pairwise separated, and their distance from the
external support is bounded below by $d_0/4$.
\end{proposition}

\begin{proof}
\emph{Step 1: Bootstrap setup.}

The global packet estimate below has a constant $C_{\rm p}$ depending only
on the fixed quantities in \eqref{eq:fixed-macro-parameters}; in particular,
it is independent of any support-bootstrap radius.  Let $C_{\rm tr}$ be the
constant in Lemma~\ref{lem:flow-comparison} for the background on
$[0,\tau]$, and choose once and for all
\begin{equation}\label{eq:Lstar-choice}
 L_*=2C_{\rm tr}\bigl(R_0+\tau(C_{\rm p}+1)\bigr)+2.
\end{equation}
For small $r$, define $T_r\in(0,\tau]$ to be the supremum of times
$T'\le\tau$ for which the following assertions hold for every
$0\le t\le T'$:
\begin{align}
 &\supp\Theta_j(t)\subset X_{q_j}(t)+rB_{L_*}
 \quad(1\le j\le J),
 \notag\\
 &\supp\Theta_{\rm e}(t)\subset\Nbh_{\kappa}(\cZ_t),
 \label{eq:bootstrap-external-support}\\
 &\|\Theta_{\rm e}(t)\|_{H^{s_*}}
 +\|\Theta_{\rm e}(t)\|_{L^1}\le1,
 \label{eq:bootstrap-external-norm}
\end{align}
At $t=0$ these properties hold if
$L_*>R_0$, because $\Theta_{\rm e}(0)=0$.  Thus $T_r>0$.

For all sufficiently small $r$, the bootstrap packet tubes lie in
$\Nbh_{4\kappa}(\cP_t)$ and are separated from the external tube:
\begin{equation}\label{eq:bootstrap-separation}
 \dist\left(
 X_{q_j}(t)+rB_{L_*+2},
 \Nbh_{2\kappa}(\cZ_t)
 \right)\ge\frac{d_0}{4}.
\end{equation}
They are also pairwise separated because
\begin{equation*}
 2r(L_*+2)<\frac{d_J}{2}.
\end{equation*}
These restrictions will be included in the final admissible scale.

\medskip
\noindent\emph{Step 2: Packet-field estimates.}

The transported packet supports remain disjoint under the common flow.  By
\eqref{eq:packet-L1}, \eqref{eq:packet-Linfty}, and
Proposition~\ref{prop:BS-toolbox}\textup{(i)},
\begin{align}
 \|v_{\rm p}(t)\|_{L^\infty}
 &\le C\|\Theta_{\rm p}(t)\|_{L^1}^{1/2}
          \|\Theta_{\rm p}(t)\|_{L^\infty}^{1/2}\notag\\
 &\le C\bigl(Jar^2\|\zeta_0\|_{L^1}\bigr)^{1/2}
          \bigl(a\|\zeta_0\|_{L^\infty}\bigr)^{1/2}
 \le C_{\rm p}r.
 \label{eq:vp-global}
\end{align}
The constant $C_{\rm p}$ depends only on the fixed quantities in
\eqref{eq:fixed-macro-parameters}, but not on $r$ or on the bootstrap radius.
Thus the choice \eqref{eq:Lstar-choice} is non-circular.

The same field is much smaller on the external tube.  Each $\Theta_j$ has
zero circulation, $L^1$ size $ar^2\|\zeta_0\|_{L^1}$, and support radius at
most $L_*r$.  By \eqref{eq:bootstrap-separation} and
Proposition~\ref{prop:BS-toolbox}\textup{(ii)}, for $0\le m\le s_*+1$,
\begin{align}
 \|D^mv_{\rm p}(t)\|_{L^\infty(\Nbh_{2\kappa}(\cZ_t))}
 &\le
 C_{m,d_0,L_*}
 \sum_{j=1}^J r\|\Theta_j(t)\|_{L^1}\notag\\
 &\le C_{\rm far}Jar^3\|\zeta_0\|_{L^1}
 \le C_{\rm far}r^3.
 \label{eq:vp-external-Cm}
\end{align}
This is the first occurrence of the full-space $O(r^3)$ gain.  It uses one
power of $r$ from zero circulation and two from the physical area of a
packet.

\medskip
\noindent\emph{Step 3: External response and closure.}

Equation~\eqref{eq:theta-e} has zero initial data, and its source is
supported in $\supp\nabla\bar\omega(t)=\cZ_t$.  Integrating along the
characteristics of the full velocity $\bar u+v$ gives
\begin{equation}\label{eq:theta-e-support-formula}
 \supp\Theta_{\rm e}(t)
 \subset
 \bigcup_{0\le\sigma\le t}
 \Phi(t,\sigma,\cZ_\sigma).
\end{equation}
For $z\in\cZ_\sigma$, compare the full trajectory
$Y(t)=\Phi(t,\sigma,z)$ with the background trajectory
$\bar Y(t)=\bar\Phi(t,\sigma,z)$.  Since
$\bar Y(t)\in\cZ_t$, Lemma~\ref{lem:flow-comparison} gives
\begin{equation}\label{eq:external-flow-comparison}
 \dist(Y(t),\cZ_t)
 \le C_{\rm flow}
 \int_\sigma^t\|v(s)\|_{L^\infty}\,\dd s.
\end{equation}
This estimate will close the support bootstrap once the $O(r)$ global
velocity bound has been proved.

Work on the bootstrap interval $[0,T_r]$.  By
\eqref{eq:bootstrap-external-support}, the unknown $\Theta_{\rm e}$ is
supported in a fixed compact tube.  The total perturbation vorticity and
every packet have zero integral, so
\begin{equation*}
 \int\Theta_{\rm e}(t,x)\,\dd x=0.
\end{equation*}
Equation \eqref{eq:theta-e} has precisely the form
\eqref{eq:abstract-response} with
\[
 q=\Theta_{\rm e},
 \qquad g=v_{\rm p},
 \qquad K*q=v_{\rm e}.
\]
The local smallness hypothesis in Lemma~\ref{lem:localized-response} follows
from \eqref{eq:vp-external-Cm}, with $\mu=C_{\rm far}r^3$.  Hence, after
requiring $r$ small enough that $\mu\le\mu_0$,
\begin{equation*}
 \sup_{0\le t\le T_r}
 \left(
 \|\Theta_{\rm e}(t)\|_{H^{s_*}}
 +\|\Theta_{\rm e}(t)\|_{L^1}
 \right)
 \le C_{\rm ext}r^3.
\end{equation*}
This is \eqref{eq:theta-e-small} on the bootstrap interval.  By
Proposition~\ref{prop:BS-toolbox}\textup{(iii)},
\begin{equation}\label{eq:ve-O-r3}
 \sup_{0\le t\le T_r}
 \|v_{\rm e}(t)\|_{H^{s_*+1}}
 \le C_{\rm ext}r^3,
\end{equation}
and Sobolev embedding gives the second term in \eqref{eq:ve-small}.
Choose $r$ smaller, if necessary, so that
$C_{\rm ext}r^2\le1$.  Combining \eqref{eq:vp-global} and
\eqref{eq:ve-O-r3} then gives the non-circular bound
\begin{equation}\label{eq:v-global-O-r}
 \sup_{0\le t\le T_r}\|v(t)\|_{L^\infty}
 \le (C_{\rm p}+1)r.
\end{equation}

We next improve packet localization.  Let
$Y(t)=\Phi(t,0,q_j+r y_0)$ with $y_0\in\supp\zeta_0$.  Compare it with
$X_{q_j}(t)$.  Lemma~\ref{lem:flow-comparison},
\eqref{eq:v-global-O-r}, and \eqref{eq:Lstar-choice} yield
\begin{align*}
 |Y(t)-X_{q_j}(t)|
 &\le C_{\rm tr}
 \left(rR_0+\int_0^t(C_{\rm p}+1)r\,\dd s\right)\\
 &\le C_{\rm tr}(R_0+\tau(C_{\rm p}+1))r
 <\frac{L_*r}{2}.
\end{align*}
Thus the packet-support bootstrap is strictly improved without choosing
$L_*$ from a constant that depends on $L_*$.

For the external support, \eqref{eq:external-flow-comparison} and
\eqref{eq:v-global-O-r} give
\begin{equation*}
 \dist\bigl(\Phi(t,\sigma,z),\cZ_t\bigr)
 \le C_{\rm flow}\tau(C_{\rm p}+1)r.
\end{equation*}
Choose $r$ so that the right-hand side is at most $\kappa/2$.  Then
\eqref{eq:theta-e-support-formula} improves
\eqref{eq:bootstrap-external-support} to
\eqref{eq:external-support-final}.
Finally, choose $r$ so that
$C_{\rm ext}r^3<1/2$; this strictly improves
\eqref{eq:bootstrap-external-norm}.  All bootstrap inequalities are thus
strictly improved at $T_r$.  By continuity, $T_r=\tau$.

Set $R_*=L_*$.  Taking $r$ additionally below
$d_J/(4(R_*+2))$ and below the geometric tube scale gives pairwise packet
separation and the fixed positive distance from the external support.
All estimates now hold on $[0,\tau]$.
\end{proof}

\subsection{Moving coordinates and rescaled error fields}
\label{sec:moving-coordinates}

Fix the parameters of Proposition~\ref{prop:macro}.  For each packet label
$q_j$, define
\begin{equation}\label{eq:rescaled}
 \eta_{j,r}(t,y)
 =\frac1a\Theta_j(t,X_{q_j}(t)+ry),
 \qquad
 W_{j,r}(t,y)
 =\frac{v(t,X_{q_j}(t)+ry)}{ar}.
\end{equation}
By \eqref{eq:packet-support},
\begin{equation*}
 \supp\eta_{j,r}(t)\subset B_{R_*}
 \qquad(0\le t\le\tau).
\end{equation*}
Differentiating \eqref{eq:rescaled}, using
$\dot X_{q_j}=\bar u(t,X_{q_j})$, and using the transport equation for
$\Theta_j$, we obtain the exact rescaled equation
\begin{equation}\label{eq:eta-raw}
 \partial_t\eta_{j,r}
 +\left[
 \frac{\bar u(t,X_{q_j}+ry)-\bar u(t,X_{q_j})}{r}
 +aW_{j,r}
 \right]\cdot\nabla_y\eta_{j,r}=0,
 \qquad
 \eta_{j,r}(0)=-\Delta f.
\end{equation}

\medskip
\noindent\emph{Background linearization.}

Taylor's formula gives
\begin{equation}\label{eq:bg-taylor}
 \frac{\bar u(t,X_{q_j}+ry)-\bar u(t,X_{q_j})}{r}
 =M_{q_j}(t)y+rG_{j,r}(t,y),
\end{equation}
where
\begin{equation*}
 G_{j,r}(t,y)
 =\int_0^1(1-\theta)
 D^2\bar u(t,X_{q_j}+\theta ry)[y,y] \,\dd\theta.
\end{equation*}
For every fixed $\mathcal R>0$ and integer $m\ge0$,
\begin{equation}\label{eq:G-bound}
 \max_j\|G_{j,r}\|_{C([0,\tau];C^m(B_{\mathcal R}))}
 \le C_{\rm bg}(\mathcal R,m).
\end{equation}
Furthermore,
\begin{equation*}
 \diver_y\bigl(rG_{j,r}(t,y)\bigr)=0,
\end{equation*}
because both the full difference quotient in \eqref{eq:bg-taylor} and the
linear field $M_{q_j}y$ are divergence-free.

\medskip
\noindent\emph{Velocity decomposition.}

The homogeneity $K(rz)=r^{-1}K(z)$ gives an exact identity for the self-field:
\begin{equation*}
 \frac{K*\Theta_j(t,X_{q_j}+r\,\cdot)}{ar}
 =K*\eta_{j,r}.
\end{equation*}
Write
\begin{equation}\label{eq:W-decomp}
 W_{j,r}=K*\eta_{j,r}+E_{j,r}^{\rm cross}+E_{j,r}^{\rm ext},
\end{equation}
where
\begin{align*}
 E_{j,r}^{\rm cross}(t,y)
 &=\frac1{ar}\sum_{k\ne j}
 (K*\Theta_k)(t,X_{q_j}(t)+ry),
 \\
 E_{j,r}^{\rm ext}(t,y)
 &=\frac1{ar}(K*\Theta_{\rm e})(t,X_{q_j}(t)+ry).
\end{align*}

Choose once and for all
\begin{equation}\label{eq:mathcalR-choice}
 \mathcal R>R+R_*+4.
\end{equation}
For sufficiently small $r$, every point in
$X_{q_j}(t)+rB_{\mathcal R}$ stays at a fixed positive distance from the
other packet supports and from $\supp\Theta_{\rm e}(t)$.

\begin{lemma}[Rescaled cross-packet field]
\label{lem:cross-rescaled}
For every integer $m\ge0$,
\begin{equation*}
 \max_j
 \|E_{j,r}^{\rm cross}\|_{C([0,\tau];C^m(B_{\mathcal R}))}
 \le C_{\rm cross}(m,\mathcal R)r^2.
\end{equation*}
\end{lemma}

\begin{proof}
For $k\ne j$, the source $\Theta_k$ has zero circulation, support radius
$R_*r$, and $L^1$ norm $ar^2\|\zeta_0\|_{L^1}$.
Proposition~\ref{prop:BS-toolbox}\textup{(ii)} gives
\begin{equation*}
 \|D_x^mK*\Theta_k\|_{L^\infty(X_{q_j}+rB_{\mathcal R})}
 \le C_m ar^3.
\end{equation*}
Since $D_y^m=r^mD_x^m$ after composition with
$x=X_{q_j}+ry$, division by $ar$ yields $C_mr^{m+2}$.  This is bounded by
$C_mr^2$ for $0<r\le1$.  Summing over finitely many $k$ proves the claim.
\end{proof}

\begin{lemma}[Rescaled external field]
\label{lem:external-rescaled}
For every integer $m\ge0$,
\begin{equation}\label{eq:Eext-bound}
 \max_j
 \|E_{j,r}^{\rm ext}\|_{C([0,\tau];C^m(B_{\mathcal R}))}
 \le C_{\rm ext,res}(m,\mathcal R)a^{-1}r^2.
\end{equation}
In particular, since $a>0$ is fixed before $r\to0$, this is an $O(r^2)$
bound.
\end{lemma}

\begin{proof}
The external support is at a fixed positive distance from
$X_{q_j}+rB_{\mathcal R}$.  Proposition~\ref{prop:BS-toolbox}\textup{(ii)} and
\eqref{eq:theta-e-small} imply
\begin{equation*}
 \|D_x^mK*\Theta_{\rm e}\|_{L^\infty(X_{q_j}+rB_{\mathcal R})}
 \le C_m\|\Theta_{\rm e}\|_{L^1}
 \le C_mr^3.
\end{equation*}
After composition, $m$ derivatives in $y$ supply $r^m$; division by $ar$
gives $C_ma^{-1}r^{m+2}\le C_ma^{-1}r^2$.
\end{proof}

Set
\begin{equation*}
 E_{j,r}=E_{j,r}^{\rm cross}+E_{j,r}^{\rm ext}.
\end{equation*}
For fixed $a$, Lemmas~\ref{lem:cross-rescaled} and
\ref{lem:external-rescaled} give
\begin{equation}\label{eq:E-r2}
 \max_j\|E_{j,r}\|_{C([0,\tau];C^m(B_{\mathcal R}))}
 \le C_Er^2.
\end{equation}
The field $E_{j,r}$ is divergence-free in $y$, because it is a rescaled sum
of divergence-free velocity fields.

The total coefficient error is
\begin{equation}\label{eq:F-def}
 F_{j,r}=rG_{j,r}+aE_{j,r}.
\end{equation}
By \eqref{eq:G-bound} and \eqref{eq:E-r2},
\begin{equation*}
 \max_j\|F_{j,r}\|_{C([0,\tau];C^m(B_{\mathcal R}))}
 \le C_Fr.
\end{equation*}
Moreover,
\begin{equation*}
 \diver_yF_{j,r}=0.
\end{equation*}
Substituting \eqref{eq:bg-taylor}, \eqref{eq:W-decomp}, and
\eqref{eq:F-def} into \eqref{eq:eta-raw} gives
\begin{equation}\label{eq:eta-error}
 \partial_t\eta_{j,r}
 +\bigl(M_{q_j}(t)y+aK*\eta_{j,r}+F_{j,r}\bigr)
 \cdot\nabla\eta_{j,r}=0,
 \qquad
 \eta_{j,r}(0)=-\Delta f.
\end{equation}

\begin{remark}[Error hierarchy]
The complete hierarchy in the rescaled velocity is
\[
 \underbrace{rG_{j,r}}_{O(r)}
 \quad+\quad
 \underbrace{aE_{j,r}^{\rm cross}}_{O(ar^2)}
 \quad+\quad
 \underbrace{aE_{j,r}^{\rm ext}}_{O(r^2)}.
\]
Thus the moving-packet approximation is limited by the first-order Taylor
remainder of the background, not by nonlocal packet interactions.
\end{remark}

The following elementary extension will allow us to apply global Sobolev
estimates to the locally defined error field $F_{j,r}$.

\begin{lemma}[Divergence-free extension from a ball]
\label{lem:df-extension}
Let $0<R_1<R_2$ and let
$E\in C^\infty(B_{R_2};\R^2)$ satisfy $\diver E=0$.  For every integer
$m\ge0$, there is a linear extension operator producing
$\widetilde E\in C_c^\infty(\R^2;\R^2)$ such that
\begin{equation*}
 \diver\widetilde E=0,
 \qquad
 \widetilde E=E\quad\text{on }B_{R_1},
 \qquad
 \|\widetilde E\|_{C^m}
 \le C_{R_1,R_2,m}\|E\|_{C^m(B_{R_2})}.
\end{equation*}
The construction applies uniformly to smooth time-dependent families.
\end{lemma}

\begin{proof}
Since $B_{R_2}$ is simply connected, there exists a stream function
$\psi$, normalized by $\psi(0)=0$, with $E=\nabla^\perp\psi$.  Integration
along radial segments controls $\psi$ in $C^{m+1}$ by $E$ in $C^m$.
Choose $\chi\in C_c^\infty(B_{R_2})$ equal to one on $B_{R_1}$ and set
$\widetilde E=\nabla^\perp(\chi\psi)$.
\end{proof}

\subsection{Quantitative moving-packet decoupling}
\label{sec:decoupling}

We compare \eqref{eq:eta-error} with the nonlinear affine profile
\eqref{eq:nonlinear-profile} corresponding to the same label $q_j$ and
amplitude $a$.

\begin{proposition}[Moving-packet approximation with $O(r)$ error]
\label{prop:decoupling}
Fix
$f,a,J,\tau,R$, the labels $q_1,\ldots,q_J$, and an integer $m\ge0$.
There exist $C_{\rm dec}<\infty$ and $r_{\rm dec}>0$, depending on these
fixed quantities and on the background but not on $r$, such that
\begin{equation}\label{eq:decoupling}
 \max_{1\le j\le J}
 \left\|
  \frac{v(t,X_{q_j}(t)+r\,\cdot)}{ar}
  -w_{q_j}^a(t,\cdot)
 \right\|_{C([0,\tau];C^m(B_R))}
 \le C_{\rm dec}r
\end{equation}
for $0<r<r_{\rm dec}$.  In particular,
\begin{equation}\label{eq:D3-decoupling}
 \max_j
 \left\|
  r^2D_x^3v(\tau,X_{q_j}(\tau)+r\,\cdot)
  -aD_y^3w_{q_j}^a(\tau)
 \right\|_{L^1(B_R)}
 \le C_{\rm dec}^{(3)}ar.
\end{equation}
\end{proposition}

\begin{proof}
Choose $s\ge m+7$.  The field $F_{j,r}$ is defined on
$B_{\mathcal R}$ and is divergence-free.  Choose radii
\[
 R_*+1<R_1<R_2<\mathcal R
\]
with $B_R\subset B_{R_1}$ if necessary by increasing $\mathcal R$.
Lemma~\ref{lem:df-extension} gives a global compactly supported
divergence-free extension, still denoted $F_{j,r}$, satisfying
\begin{equation}\label{eq:F-global-bound}
 \max_j\|F_{j,r}\|_{C([0,\tau];C^s(\R^2))}
 \le C_Fr.
\end{equation}
Because $\eta_{j,r}$ is supported in $B_{R_*}$, this extension does not
alter its equation.

Let
\begin{equation*}
 d_{j,r}=\eta_{j,r}-\zeta_{q_j}^a.
\end{equation*}
Subtracting \eqref{eq:nonlinear-profile} from \eqref{eq:eta-error}, we find
\begin{align}
 \partial_td_{j,r}
 &+\bigl(M_{q_j}y+aK*\eta_{j,r}+F_{j,r}\bigr)
   \cdot\nabla d_{j,r}\notag\\
 &=-\bigl(aK*d_{j,r}+F_{j,r}\bigr)
   \cdot\nabla\zeta_{q_j}^a,
 \qquad d_{j,r}(0)=0.
 \label{eq:d-equation}
\end{align}
Both $\eta_{j,r}$ and $\zeta_{q_j}^a$ have zero mass and support in one
fixed ball; hence so does $d_{j,r}$.  By
Proposition~\ref{prop:BS-toolbox}\textup{(iii)},
\begin{equation}\label{eq:Kd-estimate}
 \|K*d_{j,r}\|_{H^s}
 \le C\bigl(\|d_{j,r}\|_{H^{s-1}}+\|d_{j,r}\|_{L^1}\bigr)
 \le C\|d_{j,r}\|_{H^{s-1}},
\end{equation}
where the last inequality uses the common compact support.

Let $T_{j,r}$ be the maximal time on which
\begin{equation}\label{eq:d-bootstrap}
 \|d_{j,r}(t)\|_{H^{s-1}}\le1.
\end{equation}
On this interval,
$\eta_{j,r}=\zeta_{q_j}^a+d_{j,r}$ is controlled in $H^{s-1}$ by a fixed
constant.  The linear drift $M_{q_j}(t)y$ has constant spatial gradient,
and all profile norms are fixed by
Lemma~\ref{lem:profile-smooth-bounds}.  Applying the differentiated transport estimate to
\eqref{eq:d-equation}, using \eqref{eq:Kd-estimate} and
\eqref{eq:F-global-bound}, first gives the squared inequality
\begin{align*}
 \frac12\frac{\dd}{\dd t}\|d_{j,r}\|_{H^{s-1}}^2
 &\le C_d\|d_{j,r}\|_{H^{s-1}}^2
 +C_d\|F_{j,r}\|_{C^s}\|d_{j,r}\|_{H^{s-1}}.
\end{align*}
Here the coefficient $C_d$ is finite on the bootstrap interval because
$\eta_{j,r}=\zeta_{q_j}^a+d_{j,r}$, the profile is uniformly smooth, and
$\|d_{j,r}\|_{H^{s-1}}\le1$.  Regularizing the norm at zero and using
\eqref{eq:F-global-bound} yields
\begin{equation*}
 \frac{\dd}{\dd t}\|d_{j,r}(t)\|_{H^{s-1}}
 \le C_d\|d_{j,r}(t)\|_{H^{s-1}}+C_dC_Fr.
\end{equation*}
The constant $C_d$ depends on $f,a,\tau,s$ and the compact label set, but is
independent of $r$ and $j$.  Gronwall's inequality gives
\begin{equation*}
 \max_j\|d_{j,r}\|_{C([0,T_{j,r}];H^{s-1})}
 \le C_*r.
\end{equation*}
Taking $r<(2C_*)^{-1}$ strictly improves the bootstrap threshold in
\eqref{eq:d-bootstrap}; hence $T_{j,r}=\tau$ for every $j$.

Proposition~\ref{prop:BS-toolbox}\textup{(iii)} and Sobolev embedding, with
$s\ge m+7$, give
\begin{equation}\label{eq:Kd-Cm}
 \max_j\|K*d_{j,r}\|_{C([0,\tau];C^m(B_R))}
 \le C_mr.
\end{equation}
By \eqref{eq:W-decomp},
\[
 W_{j,r}-w_{q_j}^a=K*d_{j,r}+E_{j,r},
\]
and \eqref{eq:E-r2} is smaller than the right-hand side of
\eqref{eq:Kd-Cm}.  This proves \eqref{eq:decoupling}.

Finally,
\[
 v(t,X_{q_j}(t)+ry)=arW_{j,r}(t,y),
\]
so
\[
 r^2D_x^3v(t,X_{q_j}(t)+ry)=aD_y^3W_{j,r}(t,y).
\]
Integrating the $C^3$ estimate over $B_R$ proves
\eqref{eq:D3-decoupling}.
\end{proof}

\subsection{Parameter dependence and admissible packet scales}
\label{sec:parameters}

The order of choices is part of the proof.  This section records all
allowed dependencies and produces one final admissible upper bound for $r$.
Fix a target triple
\begin{equation*}
 0<\tau\le\tau_0,
 \qquad \delta>0,
 \qquad M>0.
\end{equation*}

The dependencies are summarized in the following table.

\begin{longtable}{@{}p{0.16\textwidth}p{0.31\textwidth}p{0.43\textwidth}@{}}
\toprule
Quantity & Allowed dependence & Purpose \\
\midrule
\endfirsthead
\toprule
Quantity & Allowed dependence & Purpose \\
\midrule
\endhead
$\eps$ & $\delta,C_0$ & total top-order size of the perturbation \\
$Q$ & $M,\eps$ & required local passive output \\
$f,R,h$ & $\tau,Q$, background & bad compact profile, observation radius, uniform label ball \\
$a_*$ & $f,R,h,\tau$, background & nonlinear affine profile remains close to passive profile \\
$J$ & $\eps,a_*$ & makes $a=\eps/J<a_*$ \\
$q_1,\ldots,q_J$ & $J,h$ & distinct labels in the zero-vorticity ball \\
$d_J$ & labels, background, $\tau$ & minimum separation of background trajectories \\
$R_*,C_{\rm mac},r_{\rm mac}$ & $f,a,J,q_j,\tau$, background & support propagation and $O(r^3)$ response \\
$C_E,C_F$ & all preceding fixed quantities, $\mathcal R,s$ & rescaled cross/external and total coefficient errors \\
$C_{\rm dec},r_{\rm dec}$ & all preceding fixed quantities, $R$ & $O(r)$ moving-packet approximation \\
$r$ & all preceding quantities & final microscopic scale \\
\bottomrule
\caption{Allowed dependence in the multi-packet construction.}
\label{tab:dependencies}
\end{longtable}

No quantity chosen before $r$ is required to be bounded uniformly as the
nonmultiplier gain $Q$ tends to infinity.  For each final initial datum, all
such quantities are finite and fixed.

We now combine these dependencies into one admissible upper bound for the
packet scale.

Let $R_0$ be the support radius of $f$, let $R_*$ be furnished by
Proposition~\ref{prop:macro}, and let $\mathcal R$ be as in
\eqref{eq:mathcalR-choice}.  Define the geometric threshold
\begin{equation}\label{eq:r-geom}
 r_{\rm geom}
 =\min\left\{
 \frac{\rho}{2(R_0+1)},
 \frac{\kappa}{2(R_*+\mathcal R+1)},
 \frac{d_J}{8(R_*+\mathcal R+1)}
 \right\}.
\end{equation}
The first term places the initial packets in the inner material region; the
second preserves separation from the background-vorticity tube; the third
keeps packet and observation tubes disjoint.

Let $r_{\rm cost}>0$ be such that
\begin{equation*}
 C_fJa(r+r^2+r^3)\le C_0\eps
 \qquad(0<r<r_{\rm cost}).
\end{equation*}
Let $r_{\rm out}>0$ be such that the error in
\eqref{eq:D3-decoupling} is at most $aQ/4$:
\begin{equation*}
 C_{\rm dec}^{(3)}ar\le\frac{aQ}{4}
 \qquad(0<r<r_{\rm out}).
\end{equation*}
The amplitude $a$ cancels in this condition.  Finally, define
\begin{equation*}
 r_{\rm adm}
 =\min\left\{
 1,r_{\rm geom},r_{\rm cost},r_{\rm mac},r_{\rm dec},r_{\rm out}
 \right\}.
\end{equation*}
Every estimate in the construction holds simultaneously whenever
$0<r<r_{\rm adm}$.

\begin{remark}[Why $a^{-1}$ is harmless]
The rescaled external field contains a factor $a^{-1}$ in
\eqref{eq:Eext-bound}.  This causes no loss because the order of choices is
$a$ first and $r$ last.  Once $a>0$ is fixed, the coefficient
$a^{-1}r^2$ still tends to zero faster than the $O(r)$ Taylor error.
No estimate is claimed uniformly as $a\downarrow0$ and $r\downarrow0$
simultaneously.
\end{remark}

\begin{proof}[Proof of Proposition~\ref{prop:packet-realization}]
Fix all parameters in the statement.  Proposition~\ref{prop:macro} gives
uniform packet localization and separation for $0<r<r_{\rm mac}$.
Proposition~\ref{prop:decoupling}, applied with $m=3$, gives
\eqref{eq:realization-D3} with
$C_{\rm real}=C_{\rm dec}^{(3)}$.  Taking $r_{\rm real}$ below
$r_{\rm mac}$, $r_{\rm dec}$, and the geometric threshold
\eqref{eq:r-geom} makes the sets \eqref{eq:realization-sets} pairwise
disjoint and proves the proposition.
\end{proof}

\section{Norm inflation}
\label{sec:local-proof}

The nonlinear packet construction is first proved relative to one compactly
supported background.

\subsection{Inflation near the background}

\begin{theorem}[Local unboundedness near a compact background]
\label{thm:background}
There exist a divergence-free
$\bar u_0\in C_c^\infty(\R^2;\R^2)$ and $\tau_0>0$ such that, writing
$\bar u(t)=\cS_t^\infty\bar u_0$, the following holds.  For every
$0<\tau\le\tau_0$ and every $\delta,M>0$, there exists
$U_0\in C_c^\infty(\R^2;\R^2)$ with $\diver U_0=0$ such that
\begin{equation*}
 \|U_0-\bar u_0\|_{W^{3,1}}<\delta
\end{equation*}
and, more strongly,
\begin{equation*}
 \|\cS_\tau^\infty U_0-\bar u(\tau)\|_{\dot W^{3,1}}>M.
\end{equation*}
Hence the same lower bound holds for the full $W^{3,1}$ norm.
Moreover, the perturbation $U_0-\bar u_0$ is a finite sum of disjoint
compactly supported velocity packets.
\end{theorem}

\begin{proof}[Proof of Theorem~\ref{thm:background}]
Fix $0<\tau\le\tau_0$ and $\delta,M>0$.
Choose
\begin{equation*}
 0<\eps<\min\left\{1,\frac{\delta}{4C_0}\right\},
\end{equation*}
where $C_0$ is from Lemma~\ref{lem:initial-cost}, and then choose
\begin{equation*}
 Q>\frac{8M}{\eps}.
\end{equation*}
Proposition~\ref{prop:uniform-passive} gives a compactly supported real
profile $f$, an observation radius $R$, and a label radius $h$ satisfying
\eqref{eq:uniform-passive}.  Proposition~\ref{prop:small-a} then gives
$a_*>0$ and the nonlinear output bound \eqref{eq:nonlinear-output}.

Choose $J$ so large that
\begin{equation*}
 a=\frac{\eps}{J}<a_*.
\end{equation*}
Select distinct labels $q_1,\ldots,q_J\in B_h$.  Once these labels have
been fixed, $d_J>0$ in \eqref{eq:trajectory-separation} is fixed.

Apply Proposition~\ref{prop:packet-realization} and take $r>0$ small enough
that its conclusions hold,
\begin{equation}\label{eq:realization-output-choice}
 C_{\rm real}r\le\frac Q4,
\end{equation}
and the lower-order term in \eqref{eq:initial-cost} is at most $C_0\eps$.
Lemma~\ref{lem:initial-cost} then gives
\[
 \|U_0^{J,r}-\bar u_0\|_{W^{3,1}}
 \le 2C_0\eps<\delta.
\]

Changing variables on each of the disjoint sets $\Omega_j$ in
\eqref{eq:realization-sets} and using \eqref{eq:nonlinear-output},
\eqref{eq:realization-D3}, and \eqref{eq:realization-output-choice}, we obtain
\begin{align*}
 \int_{\Omega_j}|D_x^3v(\tau,x)|\,\dd x
 &\ge a\int_{B_R}|D^3w_{q_j}^a(\tau,y)|\,\dd y
      -C_{\rm real}ar\\
 &\ge\frac{aQ}{2}.
\end{align*}
Consequently,
\begin{equation*}
 \|\cS_\tau^\infty U_0^{J,r}-\bar u(\tau)\|_{\dot W^{3,1}}
 \ge\sum_{j=1}^J\frac{aQ}{2}
 =\frac{\eps Q}{2}>M.
\end{equation*}
This proves the theorem.
\end{proof}

\subsection{Scaling to the zero solution}
\label{sec:scaling}

The Euler equations possess the amplitude--time symmetry
\begin{equation*}
 u^\lambda(t,x)=\lambda u(\lambda t,x).
\end{equation*}
Thus, for smooth data,
\begin{equation}\label{eq:solution-map-scaling}
 \cS_T^\infty(\lambda u_0)
 =\lambda\cS_{\lambda T}^\infty u_0.
\end{equation}
No spatial rescaling is involved, so the $W^{3,1}$ norm is multiplied
exactly by $\lambda$.

\begin{proof}[Proof of Theorem~\ref{thm:zero}]
Fix $T,\delta,M>0$.  Put
\[
 B_0=\|\bar u_0\|_{W^{3,1}}+1.
\]
Choose $0<\tau\le\tau_0$ so small that, with
\begin{equation*}
 \lambda=\frac{\tau}{T},
\end{equation*}
one has
\begin{equation}\label{eq:lambda-small-data}
 \lambda B_0<\delta.
\end{equation}
Apply Theorem~\ref{thm:background} at time $\tau$ with initial radius $1$
and output target
\begin{equation*}
 M_\lambda=\frac{2M}{\lambda}
 +2\|\bar u(\tau)\|_{\dot W^{3,1}}+1.
\end{equation*}
We obtain a compactly supported smooth divergence-free $U_0$ such that
\[
 \|U_0-\bar u_0\|_{W^{3,1}}<1,
 \qquad
 \|\cS_\tau^\infty U_0-\bar u(\tau)\|_{\dot W^{3,1}}>M_\lambda.
\]
Set
\begin{equation*}
 u_0=\lambda U_0.
\end{equation*}
Then \eqref{eq:lambda-small-data} gives
\[
 \|u_0\|_{W^{3,1}}\le\lambda B_0<\delta.
\]
By \eqref{eq:solution-map-scaling},
\begin{align*}
 \|\cS_T^\infty u_0\|_{\dot W^{3,1}}
 &=\lambda\|\cS_\tau^\infty U_0\|_{\dot W^{3,1}}\\
 &\ge\lambda
 \left(
 \|\cS_\tau^\infty U_0-\bar u(\tau)\|_{\dot W^{3,1}}
 -\|\bar u(\tau)\|_{\dot W^{3,1}}
 \right)>M.
\end{align*}
This proves the theorem.
\end{proof}

In particular, for every $T>0$ there is no $W^{3,1}$ neighborhood
$\mathcal U$ of zero and no continuous map
\[
 \widetilde\cS_T:\mathcal U\longrightarrow W^{3,1}(\R^2)
\]
that agrees with $\cS_T^\infty$ on compactly supported smooth
divergence-free data: such a map would be locally bounded at zero,
contradicting \eqref{eq:local-unbounded-zero}.

\section{Localized insertion and compactly supported condensation}
\label{sec:compact-insertion}

The preceding construction used one fixed hyperbolic background.  We now
show that the same finite packet mechanism can be created inside any ball
on which a prescribed smooth datum vanishes.  The time is chosen after the
datum and the ball.  This flexibility is what permits all stages of the
condensation to remain in one fixed compact set and leads directly to
Theorem~\ref{thm:compact-main}.

\subsection{Short-time insertion into an empty ball}

\begin{proposition}[Short-time insertion into an empty ball]
\label{prop:local-insertion}
Let $g_0\in C_{c,\sigma}^\infty(\R^2)$, and let $B$ be an open ball such
that
\begin{equation*}
 \overline B\cap\supp g_0=\varnothing.
\end{equation*}
For every $\eta,L,T_*>0$, there exist $0<T<T_*$ and
$h_0\in C_{c,\sigma}^\infty(B)$ satisfying
\begin{equation}\label{eq:local-insertion-conclusion}
 \|h_0\|_{W^{3,1}}<\eta,
 \qquad
 \|\cS_T^\infty(g_0+h_0)\|_{\dot W^{3,1}}>L.
\end{equation}
\end{proposition}

\begin{proof}
\emph{Step 1: a local zero-vorticity hyperbolic tube.}
Choose $z\in B$ and $s>0$ such that $z+B_{8s}\Subset B$.  Let
$\chi\in C_c^\infty(B_8)$ be radial and equal to one on $B_4$, and set
\begin{equation*}
 \psi_{z,s}(x)=\chi\!\left(\frac{x-z}{s}\right)
 (x_1-z_1)(x_2-z_2),
 \qquad b_{z,s}=\nabla^\perp\psi_{z,s}.
\end{equation*}
Fix $\beta>0$ so small that
\begin{equation}\label{eq:local-beta-small}
 \beta\|b_{z,s}\|_{W^{3,1}}<\eta/2,
\end{equation}
and define
\begin{equation}\label{eq:local-augmented-background}
 \bar u_0=g_0+\beta b_{z,s}.
\end{equation}
Because $g_0$ vanishes on $B$, on $z+B_{4s}$ we have
\begin{equation}\label{eq:local-hyperbolic-core}
 \bar u_0(x)=\beta(x_1-z_1,-(x_2-z_2)),
 \qquad \bar\omega_0(x)=0.
\end{equation}

Let $\bar u(t)=\cS_t^\infty\bar u_0$, let $\bar\Phi$ be its flow, and put
$A(t)=D\bar\Phi(t,0,z)$.  Taylor expansion at $t=0$ gives
\begin{equation*}
 A(t)^TA(t)
 =I+2\beta t
 \begin{pmatrix}1&0\\0&-1\end{pmatrix}+O(t^2).
\end{equation*}
Hence there is $T_0>0$ such that $A(t)$ is nonconformal for
$0<t\le T_0$.  Moreover,
\[
 \dist\bigl(z+\overline{B_{2s}},
       \supp\nabla\bar\omega_0\bigr)>0.
\]
Indeed, the contribution of $g_0$ is outside $B$, while the vorticity of
$b_{z,s}$ is supported in the cutoff annulus.  Vorticity transport,
compactness, and Lemma~\ref{lem:flow-comparison} imply, after decreasing
$T_0$, that for some $d_0>0$,
\begin{equation}\label{eq:local-material-separation}
 \inf_{0\le t\le T_0}
 \dist\left(
  \bar\Phi(t,0,z+\overline{B_{2s}}),
  \supp\nabla\bar\omega(t)
 \right)\ge d_0.
\end{equation}
The flow is also uniformly bi-Lipschitz on the relevant compact sets.
These are precisely the geometric properties isolated in
Proposition~\ref{prop:background}, now verified for the augmented
background \eqref{eq:local-augmented-background}.

Choose
\begin{equation*}
 0<T<\min\{T_*,T_0\}.
\end{equation*}

\emph{Step 2: passive and nonlinear affine profiles.}
For $q\in z+B_{2s}$ set
\begin{equation*}
 X_q(t)=\bar\Phi(t,0,q),
 \qquad A_q(t)=D_q\bar\Phi(t,0,q),
 \qquad M_q(t)=\nabla\bar u(t,X_q(t)).
\end{equation*}
The matrices $A_q$ solve $\dot A_q=M_qA_q$, have determinant one, and
depend smoothly on $q$.  Put
\begin{equation}\label{eq:local-target-output}
 \mathcal M=L+\|\bar u(T)\|_{\dot W^{3,1}}+1.
\end{equation}
Choose
\begin{equation*}
 0<\eps<\min\{1,\eta/(8C_0)\},
 \qquad Q>8\mathcal M/\eps,
\end{equation*}
where $C_0$ is the universal constant in
Lemma~\ref{lem:initial-cost}.

Since $A_z(T)$ is nonconformal, the symbol
\begin{equation*}
 m_{A_z(T)}(\xi)=\frac{|\xi|^2}{|A_z(T)^{-T}\xi|^2}
\end{equation*}
is nonconstant.  Apply Proposition~\ref{prop:uniform-passive} to the family
$q\mapsto A_q(T)$ with base label $z$.  It gives a real
$f\in C_c^\infty(\R^2)$ satisfying
\begin{equation*}
 \|f\|_{\dot W^{4,1}}\le1
\end{equation*}
and, after shrinking the label neighborhood to $z+B_h$ with $0<h<s$, a
finite observation radius $R$ for which the following holds.  Define
\begin{equation*}
 \zeta_q^0(t,y)=(-\Delta f)(A_q(t)^{-1}y),
 \qquad w_q^0=K*\zeta_q^0.
\end{equation*}
At time $T$,
\begin{equation*}
 \int_{B_R}|D^3w_q^0(T,y)|\,\dd y>Q
 \qquad(q\in z+B_h).
\end{equation*}

For $a\ge0$ let
\begin{equation*}
 \partial_t\zeta_q^a+
 \bigl(M_q(t)y+aK*\zeta_q^a\bigr)\cdot\nabla\zeta_q^a=0,
 \qquad \zeta_q^a(0)=-\Delta f,
 \qquad w_q^a=K*\zeta_q^a.
\end{equation*}
Lemma~\ref{lem:profile-moments} gives zero mass and zero first moment for
these profiles at every time.  Lemma~\ref{lem:profile-smooth-bounds} gives
uniform compact support and high Sobolev bounds for
$q\in z+\overline{B_h}$ and $0\le a\le1$.  The proof of
Proposition~\ref{prop:small-a}, using Lemmas~\ref{lem:transport-Hs} and
\ref{lem:difference-Hs}, gives the same estimate here because $M_q$ is
smooth and uniformly bounded for $q\in z+\overline{B_h}$.  Thus
\begin{equation*}
 \sup_{q\in z+B_h}
 \|w_q^a-w_q^0\|_{C([0,T];W^{3,1}(B_R))}\le C_fa,
\end{equation*}
and there is $a_*>0$ such that
\begin{equation}\label{eq:local-nonlinear-profile-output}
 \int_{B_R}|D^3w_q^a(T,y)|\,\dd y\ge\frac{3Q}{4}
 \qquad(q\in z+B_h,\ 0<a<a_*).
\end{equation}

\emph{Step 3: physical packets and realization.}
Choose a finite integer $J$ so large that
\begin{equation*}
 a=\eps/J<a_*,
\end{equation*}
and choose distinct labels $q_1,\ldots,q_J\in z+B_h$.  For $r>0$ define
\begin{equation*}
 \Psi_0^{J,r}(x)=ar^2\sum_{j=1}^J
 f\!\left(\frac{x-q_j}{r}\right),
 \qquad p_0^{J,r}=\nabla^\perp\Psi_0^{J,r}.
\end{equation*}
For sufficiently small $r$, the packet supports are disjoint and contained
in $z+B_{2s}$.  Lemma~\ref{lem:initial-cost} gives
\begin{equation}\label{eq:local-packet-cost}
 \|p_0^{J,r}\|_{W^{3,1}}
 \le C_0Ja\|f\|_{\dot W^{4,1}}
     +C_fJa(r+r^2+r^3).
\end{equation}
Since $Ja=\eps$, after decreasing $r$ this is less than
$2C_0\eps<\eta/2$.

We next verify the hypotheses needed for the packet-realization argument on
the present background.
Let $U^{J,r}$ be the solution from $\bar u_0+p_0^{J,r}$ and set
$v=U^{J,r}-\bar u$.  The only background properties used in the
macroscopic bootstrap are smooth coefficient bounds, bi-Lipschitz control
of the flow, a material zero-vorticity tube, and its positive separation
from $\supp\nabla\bar\omega(t)$.  They were established in
\eqref{eq:local-hyperbolic-core}--\eqref{eq:local-material-separation}.

The endpoint estimate
in Proposition~\ref{prop:BS-toolbox}\textup{(i)} first gives a global
$O(r)$ perturbation velocity.  Each packet has zero circulation, support
radius $O(r)$, and $L^1$ vorticity $O(ar^2)$; hence
Proposition~\ref{prop:BS-toolbox}\textup{(ii)} gives $O(ar^3)$ for its
velocity and every fixed derivative on the background-vorticity tube.
Summing over the $J$ fixed packets gives $O(\eps r^3)$.
Lemma~\ref{lem:localized-response} and
Proposition~\ref{prop:BS-toolbox}\textup{(iii)} then give the $O(r^3)$
external vorticity and velocity bounds of Proposition~\ref{prop:macro}.
The support comparison closes by \eqref{eq:local-material-separation}.

In the moving coordinates $x=X_{q_j}(t)+ry$, the background Taylor
remainder is $O(r)$.  Lemma~\ref{lem:cross-rescaled} gives an $O(r^2)$
cross-packet coefficient, and Lemma~\ref{lem:external-rescaled} gives an
$O(a^{-1}r^2)$ external coefficient.  Since $a>0$ and $J$ have already
been fixed, these errors are smaller than the Taylor remainder.  The
divergence-free extension of Lemma~\ref{lem:df-extension} and the
high-order difference estimate used in Proposition~\ref{prop:decoupling}
therefore yield constants $C_{\rm pkt},r_{\rm pkt}>0$, depending on the
fixed objects but not on $r$, such that
\begin{equation}\label{eq:local-moving-packet-estimate}
 \max_{1\le j\le J}
 \left\|
 r^2D_x^3v(T,X_{q_j}(T)+r\,\cdot)
 -aD_y^3w_{q_j}^a(T)
 \right\|_{L^1(B_R)}
 \le C_{\rm pkt}ar
\end{equation}
whenever $0<r<r_{\rm pkt}$.  Thus the conclusion of
Proposition~\ref{prop:packet-realization} holds for the present background,
with constants independent of $r$.

Finally decrease $r$ so that the sets $X_{q_j}(T)+rB_R$ are pairwise
disjoint and $C_{\rm pkt}r\le Q/4$.  From
\eqref{eq:local-nonlinear-profile-output} and
\eqref{eq:local-moving-packet-estimate},
\begin{align}\label{eq:local-summed-output}
 \|U^{J,r}(T)-\bar u(T)\|_{\dot W^{3,1}}
 &\ge\sum_{j=1}^J
 \int_{X_{q_j}(T)+rB_R}|D^3v(T,x)|\,\dd x\notag\\
 &\ge J\frac{aQ}{2}
 =\frac{\eps Q}{2}>\mathcal M.
\end{align}
Set $p_0=p_0^{J,r}$ and $h_0=\beta b_{z,s}+p_0$.  Both summands are
supported in $B$, and \eqref{eq:local-beta-small},
\eqref{eq:local-packet-cost} give $\|h_0\|_{W^{3,1}}<\eta$.  Since
$g_0+h_0=\bar u_0+p_0$, the reverse triangle inequality,
\eqref{eq:local-target-output}, and \eqref{eq:local-summed-output} give
\[
 \|\cS_T^\infty(g_0+h_0)\|_{\dot W^{3,1}}
 >\mathcal M-\|\bar u(T)\|_{\dot W^{3,1}}>L.
\]
This proves \eqref{eq:local-insertion-conclusion}.
\end{proof}

\subsection{Endpoint flow and lower semicontinuity}
\label{sec:constructive-condensation}

To pass from smooth finite amplifiers to a single rough datum, we use the
canonical global flow on $\cX$ and a lower-semicontinuous extension of the
top-order $L^1$ velocity seminorm.  The same two tools will also be cited in
Section~\ref{sec:far-superposition} for the complementary prescribed-time
result.

\subsubsection{The global endpoint-vorticity flow on \texorpdfstring{$\cX$}{X}}

\begin{proposition}[Canonical global flow and weak endpoint continuity]
\label{prop:endpoint-flow}
For every $u_0\in\cX$, the Euler equation has a unique global finite-energy
solution $u$ such that
\begin{equation}\label{eq:endpoint-flow-class}
 \curl u\in C([0,\infty);W^{2,1}(\R^2)).
\end{equation}
Denote this solution by $\cS_tu_0=u(t)$.  For every $T>0$,
\begin{equation}\label{eq:endpoint-lipschitz-time}
 \nabla u\in L^1(0,T;L^\infty(\R^2)),
\end{equation}
and the map
\begin{equation}\label{eq:X-to-L2-continuity}
 \cS_t:\cX\longrightarrow L^2(\R^2)
\end{equation}
is continuous for every fixed $t\ge0$.  Moreover,
$t\mapsto\cS_tu_0$ is continuous from $[0,\infty)$ to $L^2$.
\end{proposition}

\begin{proof}
If $u_0\in\cX$, then
$\omega_0=\curl u_0\in W^{2,1}$.  The endpoint embedding established and
used in \cite{DanchinJeong2026} gives
$W^{2,1}\hookrightarrow \dot B^0_{\infty,1}\hookrightarrow L^\infty$
for the present $L^1$-normalized distributions, while
$W^{2,1}\subset L^1$ by definition.  Moreover,
$u_0\in W^{1,1}\hookrightarrow L^2$ in dimension two.  The finite-energy
Biot--Savart reconstruction of $\omega_0$ is exactly $u_0$: indeed,
$u_0-K*\omega_0\in L^2$ has both divergence and curl equal to zero, and
hence vanishes by the Fourier Hodge decomposition.  The standard
finite-energy Yudovich theory therefore gives a unique global Euler solution
starting from $u_0$.  Applying the global propagation theorem of Danchin
and Jeong to its vorticity yields \eqref{eq:endpoint-flow-class}.  Their
endpoint embedding and the Biot--Savart law give on every finite interval
\begin{equation*}
 \|\nabla u(t)\|_{L^\infty}
 \lesssim \|\omega(t)\|_{\dot B^0_{\infty,1}}
 \lesssim \|\omega(t)\|_{W^{2,1}},
\end{equation*}
which proves \eqref{eq:endpoint-lipschitz-time}.

Let $u_0^{(n)}\to u_0$ in $\cX$ and write
$u^{(n)}=\cS_tu_0^{(n)}$, $u=\cS_tu_0$.  Since convergence in
$W^{3,1}$ implies convergence in $L^2$, it is enough to estimate the
velocity difference $w=u^{(n)}-u$.  It satisfies
\[
 \partial_tw+u^{(n)}\cdot\nabla w+w\cdot\nabla u+\nabla q=0,
 \qquad \diver w=0.
\]
The finite-energy identity and \eqref{eq:endpoint-lipschitz-time} imply
\begin{equation*}
 \|w(t)\|_{L^2}
 \le \|w(0)\|_{L^2}
 \exp\left(\int_0^t\|\nabla u(s)\|_{L^\infty}\,\dd s\right).
\end{equation*}
Thus $u^{(n)}(t)\to u(t)$ in $L^2$, proving
\eqref{eq:X-to-L2-continuity}.  Finally, in Leray-projected form,
\[
 \partial_tu=-\mathbb P(u\cdot\nabla u),
 \qquad
 \|u\cdot\nabla u\|_{L^2}
 \le \|u\|_{L^2}\|\nabla u\|_{L^\infty}.
\]
Energy conservation and \eqref{eq:endpoint-lipschitz-time} therefore imply
$\partial_tu\in L^1(0,T;L^2)$, proving the asserted $L^2$ time continuity.
\end{proof}

\subsubsection{A lower-semicontinuous top-order functional}

For a scalar distribution $F\in\mathcal D'(\R^2)$ define
\begin{equation}\label{eq:measure-norm-distribution}
 \|F\|_{\mathcal M}
 :=\sup\left\{
 |\langle F,\varphi\rangle|:
 \varphi\in C_c^\infty(\R^2),\ \|\varphi\|_{L^\infty}\le1
 \right\}\in[0,\infty].
\end{equation}
This is finite exactly when $F$ is a finite signed Radon measure, in which
case it is its total variation.  For a vector distribution $v=(v_1,v_2)$,
set
\begin{equation}\label{eq:V3-functional}
 \mathfrak V_3(v)
 :=\sum_{j=1}^2\sum_{|\alpha|=3}
   \|\partial^\alpha v_j\|_{\mathcal M}.
\end{equation}
For $v\in W^{3,1}$, this is an equivalent coordinate form of
$\|v\|_{\dot W^{3,1}}$.

\begin{lemma}[Distributional lower semicontinuity]
\label{lem:V3-lsc}
If $v_n\to v$ in $\mathcal D'(\R^2)$, then
\begin{equation}\label{eq:V3-lsc}
 \mathfrak V_3(v)
 \le\liminf_{n\to\infty}\mathfrak V_3(v_n).
\end{equation}
Consequently, for every $t\ge0$, the map
\begin{equation}\label{eq:V3-flow-lsc}
 u_0\longmapsto\mathfrak V_3(\cS_tu_0)
\end{equation}
is lower semicontinuous on $\cX$.
\end{lemma}

\begin{proof}
For every fixed test function $\varphi$ in
\eqref{eq:measure-norm-distribution}, distributional convergence gives
\[
 \langle\partial^\alpha v_{n,j},\varphi\rangle
 \longrightarrow
 \langle\partial^\alpha v_j,\varphi\rangle.
\]
Taking absolute values, then the liminf, and finally the supremum over
$\varphi$ proves lower semicontinuity of each term in
\eqref{eq:V3-functional}.  The sum is finite, so
\eqref{eq:V3-lsc} follows.  Proposition~\ref{prop:endpoint-flow} gives
$\cS_tu_0^{(n)}\to\cS_tu_0$ in $L^2$, hence in distributions, whenever
$u_0^{(n)}\to u_0$ in $\cX$.  This proves
\eqref{eq:V3-flow-lsc}.
\end{proof}

\begin{corollary}[Time lower semicontinuity]
\label{cor:V3-time-lsc}
For every $u_0\in\cX$, the map
\begin{equation*}
 t\longmapsto\mathfrak V_3(\cS_tu_0)
\end{equation*}
is lower semicontinuous on $[0,\infty)$.  Consequently,
$\{t:\mathfrak V_3(\cS_tu_0)>N\}$ is relatively open for every $N>0$.
\end{corollary}

\begin{proof}
By Proposition~\ref{prop:endpoint-flow}, $t\mapsto\cS_tu_0$ is continuous
into $L^2$, hence into distributions.  Apply
Lemma~\ref{lem:V3-lsc}.
\end{proof}

\begin{remark}[Role of the extended functional]
The quantity $\mathfrak V_3$ is not a new phase space.  It is an auxiliary
extended-valued functional used only inside the condensation arguments.
Proposition
\ref{prop:endpoint-flow} gives continuity of the endpoint flow into $L^2$,
not into $W^{3,1}$, so the $W^{3,1}$ norm itself does not provide the open
superlevel sets needed below.  In contrast, distributional lower
semicontinuity makes
$\{u_0:\mathfrak V_3(\cS_tu_0)>N\}$ open in $\cX$.  After this topological
step, the elementary bound of $\mathfrak V_3$ by the $W^{3,1}$ norm removes
the auxiliary notation from the statement of Theorem~\ref{thm:compact-main}.
For the complementary residual result, the stronger conclusion
$\mathfrak V_3(v)=\infty$ implies $v\notin W^{3,1}$.
\end{remark}

\subsection{Compactly supported condensation}
\label{subsec:compact-condensation}

We now prove the principal theorem by reserving one compact region and
choosing a new activation time at each stage.

\begin{proof}[Proof of Theorem~\ref{thm:compact-main}]
Fix a constant $C_3\ge1$ such that
\begin{equation}\label{eq:V3-W31-comparison}
 \mathfrak V_3(v)\le C_3\|v\|_{W^{3,1}}
 \qquad (v\in W^{3,1}(\R^2;\R^2)).
\end{equation}
Since $S$ is compact and $\mathcal O$ is an open neighborhood of $S$, one
can choose an open reserve ball $G$ such that
\begin{equation*}
 \overline G\Subset\mathcal O\setminus S.
\end{equation*}
Set
\begin{equation*}
 K=S\cup\overline G.
\end{equation*}
Then $K\Subset\mathcal O$.  Choose balls $B_n\Subset G$ with pairwise
disjoint closures.

Set $u_0^{(0)}=\bar u_0$, $t_0=1$, and
\begin{equation*}
 \mathcal C_0=\overline B_{\cX}(\bar u_0,\delta/2).
\end{equation*}
We recursively construct smooth centers $u_0^{(n)}$ supported in $K$,
times
\begin{equation}\label{eq:compact-time-recursion}
 0<t_n<\min\{\tau_n,t_{n-1}/2\},
\end{equation}
and closed balls
\begin{equation}\label{eq:compact-nested-balls}
 \mathcal C_n=\overline B_{\cX}(u_0^{(n)},\rho_n),
 \qquad 0<\rho_n\le2^{-n},
\end{equation}
such that
\begin{equation}\label{eq:compact-nesting}
 \mathcal C_n\subset\operatorname{int}\mathcal C_{n-1}
 \cap\{v_0:\mathfrak V_3(\cS_{t_n}v_0)>C_3M_n\}.
\end{equation}

Assume the construction has reached stage $n-1$.  By induction,
\begin{equation*}
 \supp u_0^{(n-1)}
 \subset S\cup\bigcup_{j<n}B_j,
\end{equation*}
so the current center vanishes on $B_n$.  Apply
Proposition~\ref{prop:local-insertion} in $B_n$, taking $T_*$ equal to the
right-hand side of \eqref{eq:compact-time-recursion}.  Choose its size
parameter $\eta$ smaller than the distance from $u_0^{(n-1)}$ to
$\partial\mathcal C_{n-1}$.  Choose the finite output target $L$ large
enough, using the equivalence between $\mathfrak V_3$ and the coordinate
$\dot W^{3,1}$ seminorm on smooth fields, that the new center
\begin{equation*}
 u_0^{(n)}=u_0^{(n-1)}+h_0
\end{equation*}
satisfies
\begin{equation*}
 u_0^{(n)}\in\operatorname{int}\mathcal C_{n-1},
 \qquad
 \mathfrak V_3(\cS_{t_n}u_0^{(n)})>C_3M_n.
\end{equation*}
The superlevel set in \eqref{eq:compact-nesting} is open by
Lemma~\ref{lem:V3-lsc}.  Hence $\rho_n>0$ can be chosen so that
\eqref{eq:compact-nested-balls}--\eqref{eq:compact-nesting} hold.

The centers form a Cauchy sequence in the complete space $\cX$: if $m>n$,
then $u_0^{(m)}\in\mathcal C_n$, so
$\|u_0^{(m)}-u_0^{(n)}\|_{W^{3,1}}\le\rho_n\le2^{-n}$.
Let $u_0$ be their limit.  Since every tail lies in every preceding closed
ball, $u_0\in\mathcal C_n$ for all $n$.  Hence
\begin{equation}\label{eq:compact-V3-thresholds}
 \mathfrak V_3(\cS_{t_n}u_0)>C_3M_n
 \qquad(n\ge1).
\end{equation}
Equations \eqref{eq:V3-W31-comparison} and
\eqref{eq:compact-V3-thresholds}, with the convention that the
$W^{3,1}$ norm is infinite outside $W^{3,1}$, give
\eqref{eq:compact-main-thresholds}.

Every center is supported in the same closed set $K$.  Convergence in
$W^{3,1}$ implies distributional convergence, so $\supp u_0\subset K$.
Moreover, every increment $u_0^{(n)}-u_0^{(n-1)}$ is supported in
$\overline G$; hence
\begin{equation*}
 \supp(u_0-\bar u_0)\subset\overline G\subset K\setminus S.
\end{equation*}
Together with $u_0\in\mathcal C_0$, these inclusions prove
\eqref{eq:compact-main-data}.
Equation \eqref{eq:compact-time-recursion} gives $t_n\downarrow0$.

Fix $T,N>0$ and choose $n$ with $t_n<T$ and $M_n>N$.  By
\eqref{eq:compact-V3-thresholds} and
Corollary~\ref{cor:V3-time-lsc}, the set
\begin{equation*}
 \{t\in(0,T):\mathfrak V_3(\cS_tu_0)>C_3N\}
\end{equation*}
contains a relative open neighborhood of $t_n$ and hence has positive
measure.  On this set \eqref{eq:V3-W31-comparison} gives
$\|\cS_tu_0\|_{W^{3,1}}>N$ (again in the extended sense).  Since $N$ is
arbitrary, \eqref{eq:compact-main-essential} follows.
\end{proof}

\section{Complementary prescribed-time strong ill-posedness}
\label{sec:far-superposition}

We now prove a complementary result with a stronger conclusion in time,
at the cost of losing compact support of the limiting datum.  Its statement
comes first so that the purpose of the far-field construction is explicit.

\begin{theorem}[Bourgain--Li-type strong ill-posedness]
\label{thm:strong}
There exists a dense $G_\delta$ set $\cG\subset\cX$ with the following
property.  For every $u_0\in\cG$, let $u$ be the unique global finite-energy
Euler solution whose vorticity satisfies
\begin{equation*}
 \omega=\curl u\in C([0,\infty);W^{2,1}(\R^2)).
\end{equation*}
Then:
\begin{enumerate}[label=\textup{(\roman*)},leftmargin=2.6em]
 \item there exists a dense $G_\delta$ set
 \begin{equation}\label{eq:residual-bad-times}
  \mathcal T(u_0)\subset(0,\infty)
 \end{equation}
 containing $\mathbb Q\cap(0,\infty)$ such that
 \begin{equation}\label{eq:residual-W31-loss}
  u(t)\notin W^{3,1}(\R^2)
  \qquad\text{for every }t\in\mathcal T(u_0);
 \end{equation}
 \item on every nonempty open interval $I\subset(0,\infty)$,
 \begin{equation}\label{eq:strong-essential-inflation}
  \operatorname*{ess\,sup}_{t\in I}
  \|u(t)\|_{W^{3,1}(\R^2)}=\infty.
 \end{equation}
\end{enumerate}
In particular, $u(1/n)\notin W^{3,1}(\R^2)$ for every $n\ge1$.
Moreover, every nonempty open ball in $\cX$ contains such a datum, obtained
as a $W^{3,1}$ limit of compactly supported smooth divergence-free fields
with a prescribed geometric convergence modulus.
\end{theorem}

Theorem~\ref{thm:zero} gives local unboundedness at zero.  To apply a
diagonal construction in the complete phase space $\cX$, we first show that
the same fixed-time unboundedness occurs in every nonempty open subset of
$\cX$.  The point is that a fixed smooth norm-inflating module may be placed
arbitrarily far from a fixed smooth background.  No estimate uniform in the
size of the module's higher Sobolev norms is needed: the module is fixed
first, and the separation distance is chosen last.

For $y\in\R^2$, write
\begin{equation*}
 (\mathsf T_y f)(x):=f(x-y).
\end{equation*}

\subsection{Density of compactly supported solenoidal data}

\begin{lemma}[Solenoidal endpoint density]
\label{lem:X-density}
The space $C_{c,\sigma}^\infty(\R^2)$ is dense in $\cX$.
\end{lemma}

\begin{proof}
Let $u\in\cX$.  Since $\diver u=0$ and $\R^2$ is simply connected, there is
a distributional stream function $\psi$ such that
\begin{equation*}
 u=\nabla^\perp\psi.
\end{equation*}
For instance, with the convention
$\nabla^\perp=(\partial_2,-\partial_1)$, one may take for almost every $x$
\[
 \psi(x_1,x_2)=-\int_{-\infty}^{x_1}u_2(s,x_2)\,\dd s.
\]
The identity $\partial_2\psi=u_1$ follows from $\diver u=0$ and the
one-dimensional vanishing at infinity of $W^{1,1}$ functions.  Hence
\begin{equation}\label{eq:density-psi-derivatives}
 D^k\psi\in L^1(\R^2),\qquad 1\le k\le4.
\end{equation}

Choose $\chi\in C_c^\infty(B_2)$ with $\chi=1$ on $B_1$ and put
$\chi_R(x)=\chi(x/R)$.  Let
\[
 A_R=B_{2R}\setminus\overline{B_R},
 \qquad
 \widetilde A_R=B_{3R}\setminus\overline{B_{R/2}}.
\]
By the scale-invariant $L^1$ Poincar\'e inequality on the connected annulus
$\widetilde A_R$, there exists a constant $c_R$ such that
\begin{equation}\label{eq:annular-poincare}
 \|\psi-c_R\|_{L^1(A_R)}
 \le CR\|\nabla\psi\|_{L^1(\widetilde A_R)}.
\end{equation}
Set $\psi_R=\chi_R(\psi-c_R)$.  For $1\le\ell\le4$, Leibniz' rule and
\eqref{eq:annular-poincare} give
\begin{align*}
 \|D^\ell\psi_R-D^\ell\psi\|_{L^1}
 &\le \|D^\ell\psi\|_{L^1(\R^2\setminus B_R)}
   +C\sum_{k=1}^{\ell-1}R^{-k}
       \|D^{\ell-k}\psi\|_{L^1(A_R)}\\
 &\quad
   +CR^{-\ell}\|\psi-c_R\|_{L^1(A_R)}
 \longrightarrow0
 \qquad(R\to\infty).
\end{align*}
Indeed, every derivative of positive order in
\eqref{eq:density-psi-derivatives} has an integrable tail, and the last term
is bounded by
$CR^{1-\ell}\|\nabla\psi\|_{L^1(\widetilde A_R)}$.
Therefore $u_R=\nabla^\perp\psi_R$ is compactly supported, divergence-free,
and converges to $u$ in $W^{3,1}$.  Finally mollify $\psi_R$ inside a
slightly enlarged compact set.  Taking its perpendicular gradient produces
a sequence in $C_{c,\sigma}^\infty(\R^2)$ converging to $u$ in $W^{3,1}$.
\end{proof}

\subsection{Products of diverging translates}

\begin{lemma}[Uniform separation of Sobolev products]
\label{lem:translate-product}
Let $s>2$ be an integer and let $\mathcal K_1,\mathcal K_2$ be compact
subsets of $H^s(\R^2)$.  Then, for every unit vector $e$,
\begin{equation}\label{eq:translate-product-limit}
 \sup_{f\in\mathcal K_1,\ g\in\mathcal K_2}
 \bigl(
 \|f\,\mathsf T_{Re}g\|_{H^{s-1}}
 +\|f\,\nabla\mathsf T_{Re}g\|_{H^{s-1}}
 \bigr)
 \longrightarrow0
 \qquad(R\to\infty),
\end{equation}
whenever the displayed products are interpreted componentwise.  The same
statement holds with the two factors interchanged.
\end{lemma}

\begin{proof}
Multiplication is continuous from $H^s\times H^s$ to $H^{s-1}$, and from
$H^s\times H^{s-1}$ to $H^{s-1}$, because $s>2$ in dimension two.  Given
$\eta>0$, choose finite $\eta$-nets for $\mathcal K_1$ and $\mathcal K_2$
whose elements lie in $C_c^\infty$.  For sufficiently large $R$, every
compactly supported element of the first finite net is disjoint from every
translated element of the second, so the corresponding products vanish.
The Sobolev product estimates bound the error made in replacing the original
functions by the finite-net elements by $C\eta$, uniformly over the two
compact sets.  Letting first $R\to\infty$ and then $\eta\to0$ proves
\eqref{eq:translate-product-limit}.
\end{proof}

\subsection{Finite-time far-field decoupling}

\begin{proposition}[Far-field superposition of two smooth Euler flows]
\label{prop:far-field-superposition}
Let $g_0,h_0\in C_{c,\sigma}^\infty(\R^2)$ and let
\[
 g(t)=\cS_t^\infty g_0,
 \qquad
 h(t)=\cS_t^\infty h_0.
\]
Fix $T>0$, an integer $m\ge3$, and a unit vector $e$.  Let $U_R$ be the
classical Euler solution with initial datum
\begin{equation*}
 U_R(0)=g_0+\mathsf T_{Re}h_0.
\end{equation*}
Then
\begin{equation}\label{eq:far-superposition-convergence}
 \sup_{0\le t\le T}
 \|U_R(t)-g(t)-\mathsf T_{Re}h(t)\|_{H^m}
 \longrightarrow0
 \qquad(R\to\infty).
\end{equation}
\end{proposition}

\begin{proof}
Put $h_R=\mathsf T_{Re}h$ and $a_R=g+h_R$.  In Leray-projected form,
$a_R$ is an approximate Euler solution with defect
\begin{equation*}
 F_R=\mathbb P\bigl(g\cdot\nabla h_R+h_R\cdot\nabla g\bigr),
\end{equation*}
where $\mathbb P$ is the Leray projector.  The trajectories
\[
 \{g(t):0\le t\le T\},\qquad
 \{h(t):0\le t\le T\}
\]
are compact in $H^{m+1}$, since the two solutions are smooth and continuous
in time with values in that space.  Lemma~\ref{lem:translate-product} and
the boundedness of $\mathbb P$ on Sobolev spaces imply
\begin{equation}\label{eq:far-defect-small}
 \|F_R\|_{L^1(0,T;H^m)}\longrightarrow0.
\end{equation}

The initial norms
$\|g_0+\mathsf T_{Re}h_0\|_{H^{m+1}}$ and the corresponding vorticity
$L^\infty$ norms are bounded independently of $R$.  To make the uniformity
explicit, let $\Omega_R=\curl U_R$.  Conservation of vorticity gives
\[
 \|\Omega_R(t)\|_{L^\infty}
 =\|\curl g_0+\mathsf T_{Re}\curl h_0\|_{L^\infty}
 \le \|\curl g_0\|_{L^\infty}+\|\curl h_0\|_{L^\infty}.
\]
Combining the standard $H^{m+1}$ energy inequality with the logarithmic
two-dimensional Euler estimate \cite{BealeKatoMajda1984,MajdaBertozzi}
\[
 \|\nabla U_R\|_{L^\infty}
 \le C\Bigl(1+\|\Omega_R\|_{L^\infty}
       \log\bigl(e+\|U_R\|_{H^{m+1}}\bigr)\Bigr)
\]
and solving the resulting scalar differential inequality yields a finite
double-exponential bound depending on $T,m,g_0,h_0$ but not on $R$.
Since translations preserve Sobolev norms, the corresponding bound for
$a_R=g+h_R$ is also uniform in $R$.  Consequently,
\begin{equation}\label{eq:far-uniform-high}
 \sup_{R\ge1}\sup_{0\le t\le T}\|U_R(t)\|_{H^{m+1}}
 +\sup_{0\le t\le T}\|a_R(t)\|_{H^{m+1}}
 \le C_{T,m,g_0,h_0}.
\end{equation}
Let $z_R=U_R-a_R$.  Since $z_R(0)=0$, subtraction of the exact and
approximate equations yields
\begin{equation*}
 \partial_tz_R
 +\mathbb P\bigl(U_R\cdot\nabla z_R+z_R\cdot\nabla a_R\bigr)
 =-F_R.
\end{equation*}
The usual $H^m$ Euler difference estimate, together with
\eqref{eq:far-uniform-high}, gives
\begin{equation*}
 \frac{\dd}{\dd t}\|z_R(t)\|_{H^m}
 \le C_{T,m,g_0,h_0}\|z_R(t)\|_{H^m}
      +\|F_R(t)\|_{H^m}.
\end{equation*}
Gronwall's inequality and \eqref{eq:far-defect-small} prove
\eqref{eq:far-superposition-convergence}.
\end{proof}

\begin{remark}[Order of choices]
The constant in \eqref{eq:far-uniform-high} may depend arbitrarily on the
fixed module $h_0$, including all of its high Sobolev norms.  This is
harmless.  In applications one first chooses $h_0$ from
Theorem~\ref{thm:zero}, then chooses the separation $R$.  No uniform
comparison between a Bonami--Poornima gain and a high Sobolev norm occurs.
\end{remark}

\subsection{Nowhere local boundedness of the smooth flow}

\begin{theorem}[Local unboundedness in every endpoint neighborhood]
\label{thm:every-neighborhood}
Fix $T>0$.  For every nonempty open set $\mathcal O\subset\cX$ and every
$L>0$, there exists $U_0\in\mathcal O\cap C_{c,\sigma}^\infty(\R^2)$ such
that
\begin{equation*}
 \|\cS_T^\infty U_0\|_{\dot W^{3,1}}>L.
\end{equation*}
Equivalently, the time-$T$ classical Euler map is not locally bounded in
any nonempty open subset of $\cX$.
\end{theorem}

\begin{proof}
Choose $f\in\cX$ and $\rho>0$ such that
$B_\cX(f,\rho)\subset\mathcal O$.  By
Lemma~\ref{lem:X-density}, choose
$g_0\in C_{c,\sigma}^\infty$ with
\begin{equation}\label{eq:base-approximation}
 \|g_0-f\|_{W^{3,1}}<\rho/3.
\end{equation}
Let $g(T)=\cS_T^\infty g_0$ and set
$B=\|g(T)\|_{\dot W^{3,1}}<\infty$.  Theorem~\ref{thm:zero} gives
$h_0\in C_{c,\sigma}^\infty$ such that
\begin{equation}\label{eq:module-choice}
 \|h_0\|_{W^{3,1}}<\rho/3,
 \qquad
 \|h(T)\|_{\dot W^{3,1}}>L+B+4,
 \qquad h(T)=\cS_T^\infty h_0.
\end{equation}
Choose a bounded ball $K\subset\R^2$ so large that
\begin{equation*}
 \|h(T)\|_{\dot W^{3,1}(K)}>L+B+3.
\end{equation*}
For a unit vector $e$, let $U_R$ solve Euler with datum
$U_{0,R}=g_0+\mathsf T_{Re}h_0$ and put $K_R=Re+K$.  Translation invariance
and \eqref{eq:base-approximation}--\eqref{eq:module-choice} give
\begin{equation*}
 \|U_{0,R}-f\|_{W^{3,1}}<2\rho/3,
\end{equation*}
so $U_{0,R}\in\mathcal O$ for every $R$.

By Proposition~\ref{prop:far-field-superposition} with $m=3$, for all
sufficiently large $R$,
\begin{equation*}
 \|U_R(T)-g(T)-\mathsf T_{Re}h(T)\|_{\dot W^{3,1}(K_R)}<1,
\end{equation*}
where we used
$\|D^3z\|_{L^1(K_R)}\le |K|^{1/2}\|D^3z\|_{L^2}$.
Consequently,
\begin{align*}
 \|U_R(T)\|_{\dot W^{3,1}}
 &\ge \|U_R(T)\|_{\dot W^{3,1}(K_R)}\\
 &\ge \|h(T)\|_{\dot W^{3,1}(K)}
      -\|g(T)\|_{\dot W^{3,1}(K_R)}-1\\
 &>L+B+3-B-1>L.
\end{align*}
Taking $U_0=U_{0,R}$ proves the theorem.
\end{proof}

\subsection{Prescribed-time diagonal condensation}
\label{sec:generic-condensation}

It remains to condense the fixed-time open dense sets into one datum.  We
record the elementary nested-ball selection used for this purpose.

\begin{lemma}[Constructive nested-ball selection]
\label{lem:nested-ball-selection}
Let $Y$ be a separable complete metric space, let
$(d_\ell)_{\ell\in\N}$ be a fixed dense sequence in $Y$, and let
$(O_k)_{k\in\N}$ be open dense subsets of $Y$.  Given $y\in Y$ and
$\delta>0$, one can recursively construct indices $\ell_k,m_k\in\N$ and
closed balls
\begin{equation}\label{eq:abstract-nested-balls}
 B_k=\overline B_Y(d_{\ell_k},2^{-m_k})
 \subset\operatorname{int}(B_{k-1})\cap O_k,
 \qquad m_k\ge k,
\end{equation}
starting from $B_0=\overline B_Y(y,\delta/2)$.  The centers converge to a
point $y_*$ satisfying
\begin{equation}\label{eq:abstract-nested-limit}
 y_*\in B_Y(y,\delta)\cap\bigcap_{k=1}^\infty O_k,
 \qquad
 d_Y(d_{\ell_k},y_*)\le2^{-k}.
\end{equation}
\end{lemma}

\begin{proof}
Assume $B_{k-1}$ has been chosen.  The set
$E_k=\operatorname{int}(B_{k-1})\cap O_k$ is nonempty and open.  Let
$\ell_k$ be the least index with $d_{\ell_k}\in E_k$, and let $m_k\ge k$
be the least integer such that
$\overline B_Y(d_{\ell_k},2^{-m_k})\subset E_k$.  Density and openness show
that both indices exist and give \eqref{eq:abstract-nested-balls}.  If
$j>k$, then $d_{\ell_j}\in B_k$, so
$d_Y(d_{\ell_j},d_{\ell_k})\le2^{-m_k}\le2^{-k}$.  The centers are therefore
Cauchy.  Their limit $y_*$ exists by completeness and belongs to every
closed ball $B_k$.  This proves \eqref{eq:abstract-nested-limit}.
\end{proof}

\begin{proof}[Proof of Theorem~\ref{thm:strong}]
Fix $t>0$ and $N\in\N$, and define
\begin{equation*}
 \cU_{t,N}
 :=\{u_0\in\cX:\mathfrak V_3(\cS_tu_0)>N\}.
\end{equation*}
By Lemma~\ref{lem:V3-lsc}, $\cU_{t,N}$ is open in $\cX$.  It is dense by
Theorem~\ref{thm:every-neighborhood}: every nonempty open subset of $\cX$
contains a compactly supported smooth datum whose classical output has
arbitrarily large $\dot W^{3,1}$ seminorm, and for smooth outputs
$\mathfrak V_3$ is equivalent to that seminorm.  Set
\begin{equation*}
 \cG_t:=\bigcap_{N=1}^\infty\cU_{t,N}
\end{equation*}
and
\begin{equation*}
 \cG:=\bigcap_{q\in\mathbb Q_+}\cG_q,
 \qquad \mathbb Q_+:=\mathbb Q\cap(0,\infty).
\end{equation*}
Both are $G_\delta$ sets.  To prove density without invoking the Baire
category theorem, fix $\bar u_0\in\cX$ and $\delta>0$, and enumerate
\begin{equation}\label{eq:enumerate-time-thresholds}
 \mathbb Q_+\times\N=\{(q_k,N_k):k\in\N\}
\end{equation}
once and for all.  Since $\cX$ is separable and
Lemma~\ref{lem:X-density} gives the density of
$C_{c,\sigma}^\infty$, fix a dense sequence
$(d_\ell)\subset C_{c,\sigma}^\infty$.  Apply
Lemma~\ref{lem:nested-ball-selection} with
\[
 Y=\cX,\qquad O_k=\cU_{q_k,N_k},\qquad y=\bar u_0.
\]
It produces smooth centers $u_0^{(k)}=d_{\ell_k}$ whose limit satisfies
\begin{equation}\label{eq:constructed-single-datum}
 u_0^*:=\lim_{k\to\infty}u_0^{(k)}\in\cG,\qquad
 \|u_0^*-\bar u_0\|_{W^{3,1}}<\delta,\qquad
 \|u_0^{(k)}-u_0^*\|_{W^{3,1}}\le2^{-k}.
\end{equation}
Consequently $\cG$ meets every open ball, and hence is dense.  This proves
that $\cG$ is a dense $G_\delta$ subset of $\cX$ and gives the asserted
constructive approximation.

For every $u_0\in\cG_t$,
\begin{equation}\label{eq:V3-infinite-fixed-time}
 \mathfrak V_3(\cS_tu_0)=\infty.
\end{equation}
Indeed, membership in every $\cU_{t,N}$ makes
$\mathfrak V_3(\cS_tu_0)$ larger than every integer.  If $u_0\in\cG$, then
\eqref{eq:V3-infinite-fixed-time} holds at every
$q\in\mathbb Q_+$.  Since the sum in \eqref{eq:V3-functional} has finitely
many terms, at least one third derivative of $u(q)$ fails to be a finite
Radon measure.  Thus $u(q)\notin W^{3,1}$ at every positive rational time,
as asserted in part~\textup{(i)}.

It remains to prove the essential-supremum statement.  By
Proposition~\ref{prop:endpoint-flow}, $t\mapsto u(t)$ is continuous in
$L^2$ and hence in distributions.  Lemma~\ref{lem:V3-lsc} therefore shows
that
\begin{equation*}
 t\longmapsto\mathfrak V_3(u(t))
\end{equation*}
is lower semicontinuous.  Hence, for every $N\ge1$, the set
\[
 \mathcal O_N(u_0):=
 \{t>0:\mathfrak V_3(u(t))>N\}
\]
is open.  It contains every positive rational time and is therefore dense
in $(0,\infty)$.  Consequently, the set
\[
 \mathcal T(u_0):=\bigcap_{N=1}^\infty\mathcal O_N(u_0)
 =\{t>0:\mathfrak V_3(u(t))=\infty\}
\]
is a dense $G_\delta$ subset of $(0,\infty)$ containing $\mathbb Q_+$.
At every $t\in\mathcal T(u_0)$, at least one third derivative of $u(t)$ is
not a finite Radon measure; in particular, $u(t)\notin W^{3,1}$.  This proves
\eqref{eq:residual-bad-times}--\eqref{eq:residual-W31-loss}.

Let $I\subset(0,\infty)$ be a nonempty open interval and let $N\ge1$.
Choose $q\in I\cap\mathbb Q_+$.  Since
$\mathfrak V_3(u(q))=\infty$, the set
\begin{equation*}
 \{t\in I:\mathfrak V_3(u(t))>N\}
\end{equation*}
is an open neighborhood of $q$ relative to $I$, and therefore has positive
Lebesgue measure.  At every such time, either $u(t)\notin W^{3,1}$, in which
case its extended $W^{3,1}$ norm is infinite, or
$u(t)\in W^{3,1}$ and its top-order seminorm is larger than a fixed constant
multiple of $N$.  Since $N$ is arbitrary, this proves
\eqref{eq:strong-essential-inflation}.

The construction \eqref{eq:enumerate-time-thresholds}--
\eqref{eq:constructed-single-datum}, applied inside an arbitrary open ball,
proves the final constructive-density assertion.
\end{proof}

\begin{corollary}[No endpoint solution bounded in time]
\label{cor:no-Linfty-endpoint-solution}
For every $u_0\in\cG$ and every $T>0$, the canonical solution of
Proposition~\ref{prop:endpoint-flow} does not belong to
$L^\infty(0,T;W^{3,1})$.  In particular, there is no finite-energy
Yudovich-class Euler solution with the same datum that belongs to
$L^\infty(0,T;W^{3,1})$.
\end{corollary}

\begin{proof}
The first assertion is \eqref{eq:strong-essential-inflation}.  Any
finite-energy solution in $L^\infty(0,T;W^{3,1})$ has bounded vorticity and
lies in the Yudovich uniqueness class; it must agree with the canonical
solution, contradicting the first assertion.
\end{proof}

The condensation uses only finite smooth amplifiers; no infinite packet
hierarchy is constructed within a single classical solution.

\appendix

\section{Transport estimates used in the bootstraps}
\label{app:transport}

We record a standard form of the high-order transport estimate, emphasizing
the dependence used in Sections~\ref{sec:tools} and
\ref{sec:decoupling}.

\begin{lemma}[High-order transport estimate]
\label{lem:transport-Hs}
Let $s\ge3$ be an integer.  Suppose $b$ is smooth and divergence-free and
$q$ solves
\[
 \partial_tq+b\cdot\nabla q=F.
\]
If $q$ is supported in a fixed compact set $E$ and $b=b_0+b_1$, where
$b_0$ is a fixed smooth field and $b_1$ is smooth on a neighborhood of $E$,
then
\begin{align*}
 \frac{\dd}{\dd t}\|q(t)\|_{H^s}
 \le C\bigl(&1+\|\nabla b_0\|_{W^{s-1,\infty}(E_1)}
              +\|\nabla b_1\|_{W^{s-1,\infty}(E_1)}\bigr)
              \|q(t)\|_{H^s}\\
 &+C\|F(t)\|_{H^s},
\end{align*}
where $E_1$ is any fixed neighborhood of $E$.  More generally, the
$W^{s-1,\infty}$ norms may be replaced by the standard Sobolev product
norms supplied by $H^{s+1}$ control of the velocity.
\end{lemma}

\begin{proof}
Apply $D^\alpha$, $|\alpha|\le s$, multiply by $D^\alpha q$, and integrate.
The principal transport contribution vanishes because $\diver b=0$.
The remaining commutator is a finite sum of products
$D^\beta b\,D^{\alpha-\beta+1}q$.  Since all derivatives of $q$ are
supported in $E$, only the local norms of $b$ on $E_1$ enter.  The result
follows from H\"older's inequality and the usual Moser estimates; see
\cite{KatoPonce1988,BahouriCheminDanchin2011}.
\end{proof}

\begin{lemma}[Difference estimate for two vorticity equations]
\label{lem:difference-Hs}
Let $\eta$ and $\zeta$ be compactly supported zero-mass vorticities solving
\begin{align*}
 \partial_t\eta+(b+aK*\eta+F)\cdot\nabla\eta&=0,\\
 \partial_t\zeta+(b+aK*\zeta)\cdot\nabla\zeta&=0,
\end{align*}
where $b(t,y)=M(t)y$ with bounded $M$, and $F$ is smooth,
divergence-free, and compactly supported.  Suppose the two vorticities remain
in a fixed ball and $\zeta$ is bounded in $H^s$, $s\ge6$.  On any interval where
$\|\eta-\zeta\|_{H^{s-1}}\le1$,
\begin{equation*}
 \frac{\dd}{\dd t}\|\eta-\zeta\|_{H^{s-1}}
 \le C\|\eta-\zeta\|_{H^{s-1}}+C\|F\|_{C^s}.
\end{equation*}
The constant depends on the fixed profile bounds, $a$, $M$, and the common
support radius.
\end{lemma}

\begin{proof}
Set $d=\eta-\zeta$.  Subtracting the two equations gives
\[
 \partial_td+(b+aK*\eta+F)\cdot\nabla d
 =-(aK*d+F)\cdot\nabla\zeta.
\]
The difference has zero mass and fixed compact support.  Hence
Proposition~\ref{prop:BS-toolbox}\textup{(iii)} and the support bound imply
\[
 \|K*d\|_{H^s}\le C\|d\|_{H^{s-1}}.
\]
On the bootstrap interval,
$\|\eta\|_{H^{s-1}}\le\|\zeta\|_{H^{s-1}}+1$; applying the same
Biot--Savart estimate to $\eta$ therefore gives uniform Sobolev product
bounds for the transporting velocity.  Since $s\ge6$, the Moser estimates
yield
\[
 \|(aK*d+F)\cdot\nabla\zeta\|_{H^{s-1}}
 \le C\|d\|_{H^{s-1}}+C\|F\|_{C^s}.
\]
Lemma~\ref{lem:transport-Hs}, applied at order $s-1$, now gives the claimed
differential inequality.
\end{proof}

\section{Compact reduction and affine-profile estimates}
\label{app:profile-estimates}

\subsection{Compactly supported witnesses for the nonmultiplier theorem}
\label{app:compact-reduction}

\begin{lemma}[Compact-support reduction in $\dot W^{4,1}$]
\label{lem:compact-reduction}
Suppose that a degree-zero symbol $m\in C^\infty(\R^2\setminus\{0\})$
satisfies
\begin{equation}\label{eq:compact-test-bound}
 \|m(D)g\|_{\dot W^{4,1}}
 \le C\|g\|_{\dot W^{4,1}}
 \qquad\text{for all }g\in C_c^\infty(\R^2).
\end{equation}
Then the same estimate holds for every $g\in\mathcal S(\R^2)$ and $m$ is a
multiplier on the completion defining $\dot W^{4,1}$ modulo polynomials.
\end{lemma}

\begin{proof}
Choose $\vartheta\in C_c^\infty$ with $\vartheta=1$ on $B_1$ and put
$g_R(x)=\vartheta(x/R)g(x)$.  Leibniz' rule gives
\[
 D^4(g_R-g)
 =\bigl(\vartheta(x/R)-1\bigr)D^4g
 +\sum_{1\le|\beta|\le4}
 R^{-|\beta|}(D^\beta\vartheta)(x/R)D^{4-\beta}g.
\]
The first term tends to zero in $L^1$ by dominated convergence, and the
remaining terms tend to zero because they are supported where
$|x|\simeq R$ while $g$ is rapidly decreasing.  Thus
$g_R\to g$ in $\dot W^{4,1}$ and in $\mathcal S'$.

Applying \eqref{eq:compact-test-bound} to $g_R-g_S$ shows that
$\partial^\alpha m(D)g_R$ is Cauchy in $L^1$ for every $|\alpha|=4$.
Multiplication by $m$ is continuous from $\mathcal S$ to $\mathcal S'$, so
the distributional limit is $\partial^\alpha m(D)g$.  Hence
\[
 \|m(D)g\|_{\dot W^{4,1}}
 \le C\|g\|_{\dot W^{4,1}}.
\]
The same Cauchy argument gives the extension to the homogeneous completion;
four derivatives remove the polynomial ambiguity.
\end{proof}

\subsection{Uniform nonlinear affine profiles}

\begin{proof}[Proof of Lemma~\ref{lem:profile-smooth-bounds}]
Write $Y(t)$ for a characteristic of \eqref{eq:nonlinear-profile}.  The
matrices $M_q(t)$ are uniformly bounded on
$[0,\tau]\times\overline{B_h}$.  Moreover, divergence-free transport
preserves the $L^1$ and $L^\infty$ norms of $\zeta_q^a$.  Therefore
Proposition~\ref{prop:BS-toolbox}\textup{(i)} gives
\[
 \sup_{q\in\overline{B_h},\ 0\le a\le A_0}
 \|w_q^a(t)\|_{L^\infty}\le C.
\]
Consequently,
\[
 \frac{\dd}{\dd t}|Y(t)|
 \le C|Y(t)|+A_0C,
\]
and Gronwall's inequality yields the uniform support radius
\eqref{eq:profile-support}.

For the high-order bound, differentiate the transport equation through
order $s$.  Standard commutator estimates give
\begin{equation}\label{eq:profile-Hs-ode}
 \frac{\dd}{\dd t}\|\zeta_q^a(t)\|_{H^s}
 \le C\bigl(1+a\|\nabla w_q^a(t)\|_{L^\infty}\bigr)
       \|\zeta_q^a(t)\|_{H^s}.
\end{equation}
The constant is uniform in $q$ because the spatial gradient of
$M_q(t)y$ is the bounded matrix $M_q(t)$.  For compactly supported smooth
vorticity in two dimensions, the logarithmic Biot--Savart estimate gives
\[
 \|\nabla w_q^a\|_{L^\infty}
 \le C\left(1+\|\zeta_q^a\|_{L^\infty}
       \log\bigl(e+\|\zeta_q^a\|_{H^s}\bigr)\right),
\]
where the constant is uniform because the support radius, $L^1$ norm, and
$L^\infty$ norm are already uniform.  Substitution into
\eqref{eq:profile-Hs-ode} gives an Osgood inequality with coefficients
independent of $(q,a)$.  It yields \eqref{eq:profile-Hs} on $[0,\tau]$.
The case $a=0$ is included.  See also
\cite{MajdaBertozzi,Chemin1998} for the logarithmic estimate.
\end{proof}

\begin{proof}[Proof of Proposition~\ref{prop:small-a}]
Set $d_q^a=\zeta_q^a-\zeta_q^0$.  Subtracting the two profile equations
gives
\begin{equation*}
 \partial_td_q^a+M_q(t)y\cdot\nabla d_q^a
 +a w_q^a\cdot\nabla d_q^a
 =-a w_q^a\cdot\nabla\zeta_q^0,
 \qquad d_q^a(0)=0.
\end{equation*}
Choose $s\ge8$.  Lemma~\ref{lem:profile-smooth-bounds} supplies a common
support ball and uniform $H^s$ bounds for $q\in\overline{B_h}$ and
$0\le a\le1$.  Proposition~\ref{prop:BS-toolbox}\textup{(iii)} controls
$w_q^a$ in $H^{s+1}$ uniformly.  Applying the high-order transport estimate
of Lemma~\ref{lem:transport-Hs}, together with the usual Moser product
bound, gives
\begin{equation*}
 \frac{\dd}{\dd t}\|d_q^a(t)\|_{H^{s-1}}
 \le C\|d_q^a(t)\|_{H^{s-1}}+Ca.
\end{equation*}
All constants are independent of $q$ and $0<a\le1$.  Gronwall's inequality
therefore yields
\begin{equation*}
 \sup_{q\in B_h}
 \|d_q^a\|_{C([0,\tau];H^{s-1})}\le Ca.
\end{equation*}

Both profiles have zero mass by Lemma~\ref{lem:profile-moments}, so the
difference has zero mass as well.  Applying
Proposition~\ref{prop:BS-toolbox}\textup{(iii)} to $d_q^a$ and then using
Sobolev embedding and H\"older's inequality on the fixed ball $B_R$ gives
\[
 \|w_q^a-w_q^0\|_{W^{3,1}(B_R)}
 \le C_R\|K*d_q^a\|_{H^s}\le Ca.
\]
This proves \eqref{eq:small-a}.  Finally combine it with the passive lower
bound \eqref{eq:uniform-passive} and reduce $a_*$ so that the top-order
local error is at most $Q/4$; this gives \eqref{eq:nonlinear-output}.
\end{proof}

\rightline{\scshape Y{\footnotesize INGZHE} B{\footnotesize AN}}
\medskip
{\footnotesize
	\rightline{Beijing University of Technology}
	\rightline{Beijing, 100124, P.R. China}
	\rightline{(banyingzhe@bjut.edu.cn)}
}
\medskip
\rightline{\scshape M{\footnotesize AOTUO} G{\footnotesize UO}}
\medskip
{\footnotesize
	\rightline{School of Science, Harbin University of Science and Technology}
	\rightline{Harbin 150080 P. R. China}
	\rightline{(18813122370@163.com)}
}
\medskip
\rightline{and}
\medskip

\rightline{\scshape Q{\footnotesize I} Z{\footnotesize HANG}}
\medskip
{\footnotesize
	\rightline{School of Mathematical Sciences, Harbin Engineering University}
	\rightline{Harbin 150001 P. R. China}
	\rightline{(qizhangmath@hrbeu.edu.cn)}
}
\end{document}